\documentclass{amsart}
\makeatletter
\@namedef{subjclassname@2020}{%
  \textup{2020} Mathematics Subject Classification}
\makeatother
\usepackage{amssymb}
\usepackage[utf8]{inputenc} 
\usepackage[T1]{fontenc}    
\usepackage{url}            
\usepackage{booktabs}       
\usepackage{amsfonts}       
\usepackage{nicefrac}       
\usepackage{microtype}      
\usepackage{color}
\usepackage{lipsum}
\usepackage{float}
\usepackage{calrsfs}
\usepackage{enumitem}
\usepackage{tikz-cd}
\usepackage{tikz}
\usepackage{graphicx}
\usepackage{old-arrows} 
\usetikzlibrary{arrows}
\usepackage{comment}
\usepackage{bm}
\usepackage[hidelinks]{hyperref}
\DeclareMathAlphabet{\pazocal}{OMS}{zplm}{m}{n}

\usepackage{eqnarray,amsmath}
\usepackage{marginnote}

\newcommand{\C}{\mathcal{C}}

\newcommand{\MM}{\mathcal{M}}

\newcommand{\el}{\textit{l}}
\newcommand{\M}{{\rm\mathbf{M}}}

\def\f{\mathop{\mathfrak{f}}}

\DeclareMathOperator{\Z}{Z}
\DeclareMathOperator{\PP}{Path}
\DeclareMathOperator{\A}{Ass}

\DeclareMathOperator{\Imagen}{Im}
\DeclareMathOperator{\len}{length}
\DeclareMathOperator{\rad}{rad}
\DeclareMathOperator{\soc}{Soc}
\DeclareMathOperator{\lcm}{lcm}

\DeclareMathOperator{\sink}{Sink}
\DeclareMathOperator{\source}{Source}
\DeclareMathOperator{\ann}{Ann}

\def\remove#1{}

\newtheorem{lemma}{Lemma}[section]
\newtheorem{corollary}[lemma]{Corollary}
\newtheorem{theorem}[lemma]{Theorem}
\newtheorem{proposition}[lemma]{Proposition}
\newtheorem*{theorem*}{Theorem}

\theoremstyle{definition}
\newtheorem{remark}[lemma]{Remark}
\newtheorem{definition}[lemma]{Definition}

\newtheorem{example}[lemma]{Example}

\newtheorem{observation}[lemma]{Observation}
\title{Algebraic characterisations and structure of path algebras}

\author[Mart\'in]{Dolores Mart\'in Barquero} 
\address{D. Mart\'{\i}n Barquero: Departamento de Matem\'atica Aplicada, Escuela de Ingenier\'{\i}as Industriales, Universidad de M\'alaga, 
M\'alaga, Spain.}
\email{dmartin@uma.es}

\author[Mart\'in]{C\'andido Mart\'in Gonz\'alez} 
\address{C. Mart\'{\i}n Gonz\'alez:  Departamento de \'Algebra Geometr\'{\i}a y Topolog\'{\i}a, Fa\-cultad de Ciencias, Universidad de M\'alaga, 
M\'alaga, Spain.}
\email{candido\_m@uma.es}

\author[Ruiz]{Iv\'an Ruiz Campos}
\address{I. Ruiz Campos:  Departamento de \'Algebra Geometr\'{\i}a y Topolog\'{\i}a, Fa\-cultad de Ciencias, Universidad de M\'alaga, 
M\'alaga, Spain.}
\email{ivaruicam@uma.es}


\begin{document}
\xdef\Rad{2}
\newcommand{\ARW}[2][]{%
    \foreach \ang in {#2}{%
        \draw[#1] (\ang:\Rad)--(\ang+1:\Rad) ;
    }
}

\maketitle
\begin{abstract}
In this work, we focus on studying the structure of path algebras over a field associated to arbitrary graphs. We characterise perfection (simplicity, primitivity, primeness and semiprimeness) and finitness conditions (Artinianity, semiartinianity and Noetherianity) in terms of geometric conditions in the associated graph. In order to do so, we also compute the socle and the Jacobson radical of a path algebra. In addition, we study the centroid of any path algebra and the extended centroid and central closure of the path algebra of a cycle. We obtain two structure theorems, one for semiprime path algebras, and another for Noetherian ones. Semiprime path algebras are direct sum of simple, prime and primitive algebras, and Noetherian path algebras modulo its radical will be isomorphic to upper triangular formal matrix algebras, they can also be seen as direct sums of path algebras of cycles and copies of the ground field itself. 
\end{abstract}

\section{Introduction}
The theory of path algebras is usually circunscripted to the study of representations of graphs. In this ground, path algebras play a preeminent roll usually linked to finite graphs. Our work plans to deal with path algebras by themselves, with their structure and finiteness/perfection properties. Finite graphs produce unital algebras so the realm of path algebras, in most studies focused to representation theory, is that of unital algebras. When one consider arbitrary graphs, possibly infinite, the corresponding path algebra is non unital and therefore we are taken to the nonunital setting, which requires some changes (in some cases the deviation from the classical arguments is deep). Concerning perfection properties we study semiprimeness, primeness and primitivity of path algebras. It turns out that any semiprime path algebras is a direct sum of simple, cycle path algebras and primitive ones. This justifies focusing on path algebras of cycles and on primitive ones.
On the other hand interesting finiteness  conditions that we study are Noetherianity and semiartianity. Noetherian path algebras have been characterized in terms of their graphs in \cite{quivlecs} but no study on their structure has been issued. These algebras turn out to be triangular and their structure modulo its radical is completely determined in this work. On the other hand semiartinian path algebras are related to the theory of ordinals, what makes of them a somewhat slippery soil.


In the literature of the so called \emph{finitness} and \emph{perfection} conditions in algebras associated to graphs, we can find  works like \cite{lpa}, \cite{quivlecs}, \cite{prop_semiprimo_mercedes}, \cite{paulsmith},  etc. The characterisation of Noetherian path algebras appears in \cite{quivlecs} for finite acyclic graphs (quivers). However, this characterisation is also valid for arbitrary graphs (Proposition \ref{prop_noetheriano}). 
Concerning semiprimeness, M. Siles in \cite[Proposition 2.1]{prop_semiprimo_mercedes} characterizes the semiprime character of $KE$, for a row finite graph $E$, taking advantage that this algebra can be embedded in the Leavitt Path algebra $L_K(E)$. Again, this result is also true for an arbitrary graph (in fact the proof in \cite{prop_semiprimo_mercedes} is also true in the general case, mutatis mutandis). We provide a new proof not relying on Leavitt theoretic arguments.
In \cite{paulsmith} we can also find the characterisation of a prime path algebra associated to a finite graph. Nevertheless, we think that it is necessary to extend the proof in order to ensure the result for arbitrary graphs. 
As far as we know the study of semiartinianity in the context of path algebras had not been addressed in the literature. The characterization of such path algebras take us to the very heart of the theory of ordinals (which might be the reason why these path algebras have not been considered previously).

In the path algebras context, these algebraic conditions are equivalent to some geometric characteristics in the corresponding graph. This allow us to easily construct a desired algebra verifying some of this finiteness and perfection conditions. Furthermore, some of them will be related. For example, semiprimeness is related with primitivity, primeness and simplicity. A semiprime path algebra will be a direct sum of prime, primitive and simple algebras.

This paper is organised as follows:  in Section \ref{sec_preliminares} we develop all the theoretic frame that will be needed. We define what is a collapse of a graph through vertices and edges and we describe the Peirce decomposition of a path algebra, in particular $KC_n$. In Section \ref{sec_centroide}, we compute the centroid of a path algebra of a connected graph (Theorem \ref{teo_centroide}) and, since the path algebra of a cycle with $n$ vertices is the only one with a nontrivial centroid, in Subsection \ref{susec_centroideextendido} we will study its extended centroid (Proposition \ref{prop_centroide_extendido_ciclo}) and its central closure (Theorem \ref{teorema_clausura_central_ciclo}),  following the construction given in \cite{baxter_central_closure}. We can make use of this construction because, as we will prove in Section \ref{sec_prime}, the path algebra of a cycle is prime (Proposition \ref{prop_caracterizacion_prima}) and it follows that all of its ideals are essential. We will characterise Artinian and semiartinian path algebras in Section \ref{sec_artinian} (Theorem \ref{prop_semiartiniana}). In order to do so, we will study the left and right socle of a path algebra. They will be related to the sources and sinks of the graph (Corollary \ref{corolario_zocalo_pathalgebra}). Furthermore, when we study the chain of socles we will generalise the concept of a source (resp. sink) of any ordinal $\alpha$ (Definition \ref{def_ifuente}). In addition, we will provide two examples of acyclic graphs, one which its corresponding path algebra is semiartinian (Figure \ref{fig_levigraph}) and other which its path algebra is not  (Figure \ref{fig_lineainfinita}). In Section \ref{sec_prime} we will focus on characterising semiprime, prime and primitive path algebras (Proposition \ref{prop_caracterizacion_semiprima}, Proposition \ref{prop_caracterizacion_prima} and Theorem \ref{teorema_primitiva}). These perfection conditions will be related to the type of connection the graph presents. In addition, this three conditions will be related as every semiprime path algebra is a direct sum of simple, primitive and prime (but not primitive) path algebras (Theorem \ref{teorema_estructura_semiprima}). Lastly, in Section \ref{sec_noetherian} we will prove that the path algebra of a cycle is Noetherian (Lemma \ref{lemma_ciclo_noetheriano}). We will make use of this result to prove the characterisation of a left (resp. right) Noetherian path algebra (Proposition \ref{prop_noetheriano}). Not only that, we will give an structure theorem that will define the behaviour of this type of algebras (Theorem \ref{teorema_structurenoetheriana}). 

\section{Preliminaries and basic results}\label{sec_preliminares}
A \emph{directed graph} consists of a $4$-tuple $E=(E^0, E^1, s_E, r_E)$ 
with $E^0$ and $E^1$ two disjoint sets and two maps $r_E, s_E: E^1 \to E^0$. The elements of $E^0$ are the \emph{vertices} and the elements of 
$E^1$ are the \emph{edges} of $E$.  Further, for $e\in E^1$, $r_E(e)$ and $s_E(e)$ are 
called the \emph{range} and the \emph{source} of $e$, respectively. If there is no confusion with respect to the graph we are considering, we simply write $r(e)$ and $s(e)$.

If $E$ is a directed graph, a vertex $v\in E^0$ such that $s^{-1}(v) = \emptyset$ is called a \emph{sink}. A vertex such that $r^{-1}(v) = \emptyset$ is a \emph{source}. We will denote by $\sink(E)$ and $\source(E)$ the set of sinks and sources of $E$ respectively. If $s^{-1}(v)$ is a finite set  for every vertex $v \in E^0$, then the graph is called \emph{row-finite}.

Let us consider a directed graph $E$. A \emph{path of length $n \geq 1$} in $E$ is a word $\mu = f_1f_2\cdots f_n$ with $f_i \in E^1$ for all $i = 1,\ldots,n$ and $r(f_i) = s(f_{i+1})$ for every $i = 1,\ldots, n-1$. The set of paths of length $n$ is denoted by $E^n$, with $E^0$ as the set of paths of length $0$, or \textit{trivial paths}, and $E^1$ as the set of paths of length $1$.

The set of all paths of $E$ is denoted by $\PP(E)$ and, as it is obvious, $\PP(E) = \bigsqcup_{n =0}^\infty E^n$.

Once we acknowledge the concept of path, we are able to define the source and the range of a path and some interesting sets we can consider in a certain path.

Let us consider a directed graph $E = (E^0,E^1,s,r)$ and a path $\mu = f_1f_2\cdots f_n \in \PP(E)$. We can extend the applications $s,r \colon \PP(E) \to E^0$ as
\begin{equation*}
\begin{split}
        & s(\mu) := s_E(f_1), \ r(\mu):= r_E(f_n) \ \text{if} \ {\rm length}(\mu)\geq 1 \\
        & s(u) = r(u) := u, \ u \in E^0.
\end{split}
\end{equation*}
And for a path $\mu$ we define the sets $\mu^{0} = \{s(f_1), r(f_i) \colon  1 \leq i \leq n\}$ and $ \mu^{(1)} = \{f_1,f_2,\ldots,f_n\}$.

In a directed graph there are different types of paths depending of whether it ends at the same vertex it has started or if all the vertices that the paths goes through are different. We will define them in the following. 

 Let $E$ be a directed graph.
\begin{enumerate}
    \item Let us consider a path $\mu = f_1f_2\cdots f_n \in \PP(E)$ with $n \geq 1$. If $s(\mu) = r(\mu) = v$, then $\mu$ is called a \emph{closed path based in $v$}.
    \item A \emph{simple path} is a path $\mu = f_1f_2\cdots f_n$ such that $s(f_i) \neq s(f_j) $ for $i\neq j$. 
    \item A \emph{closed simple path based at $v$} is a closed path $\mu =f_1f_2\cdots f_n$ based at $v$ such that $s(f_j) \neq v$ for every $j>1$.
    \item If $\mu = f_1f_2\cdots f_n$ is a closed path  based at $v$ with $s(f_i) \neq s(f_j)$ for $i \neq j$, then $\mu$ is a \emph{cycle based at} $v$. Notice that a cycle is a closed simple path based on any of its vertices, but a simple closed path based at a certain vertex $v$ may not be a closed simple path based at another vertex.
    \item Consider the path $\mu = f_1f_2\cdots f_n$, we say that $e$ is an \emph{exit} of $\mu$ if there is $i$ ($1\leq i \leq n$) with $s(f_i) = s(e)$ and $e \neq f_i$. On the other hand, it is possible to define the dual concept. We say that $e$ is an \emph{entry} of $\mu$ if there is a $j$ ($1\leq j\leq n$) with $r(f_j) = r(e)$ and $f_j \neq e$.
    \item A directed graph $E$ is a \emph{tree} if there is no closed paths in $\PP(E)$.
    \end{enumerate}

Considering a certain graph $E$, we can construct an associative $K$-algebra with $K$ a field from it.

Let us consider a directed graph $E$ and a field $K$. We define the \emph{path algebra of $E$}, denoted by $KE$, as the $K$-algebra generated by $\PP(E)$ as a $K$-vector space with the product defined as follows:
\begin{equation*}
\left \{
\begin{array}{ll}
     u\cdot v := \delta_{u,v}u, & \text{for $u,v \in E^0$}\\
     u \cdot f := \delta_{u,s(f)}f, &  \text{for $u \in E^0$ and $f\in E^1$}\\
     f \cdot v := \delta_{r(f),v}f, & \text{for $v \in E^0$ and $f\in E^1$} \\
     f \cdot e := \delta_{r(f),s(e)}fe, &  \text{for $f,e \in E^1$}
\end{array} \right .
\end{equation*}
with $\delta_{u,v} = 0$ if $u \neq v$ and $\delta_{u,v} = 1$ if $u=v$. In general, we will denote the product by yuxtaposition. 

Given a graph $E$ and its corresponding path algebra $KE$, we can obtain new path algebras by eliminating edges and vertices in the previous graph. This process of elimination is obtained thanks to a quotient by a certain ideal. The following result is not difficult to prove and it is well known
\begin{proposition} \label{prop_colapso}
    Let $E = (E^0,E^1,s,r)$ be a directed graph and $K$ a field. Consider a subset $X \subset E^0$ and the subgraph $F = (F^0,F^1,s,r)$ with $F^0 = E^0\setminus X$ and $F^1 = E^1 \setminus (s^{-1}(X) \cup r^{-1}(X))$. If $\left (X \right )$ is the two-sided ideal in $KE$ generated by the vertices in $X$, then $KE/(X) \cong KF$ as $K$- algebras. The graph $F$ is denoted as \emph{the collapse of $E$ through the set of vertices $X$}.

Consider the subset $Y \subset E^1$ and the subgraph $G = (G^0,G^1,s,r)$ with $G^0 = E^0$ and $G^1 = E^1 \setminus Y$. If $(Y)$ is the two-sided ideal in $KE$ generated by the set $Y$ of edges, then $KE / (Y) \cong KG$ as $K$-algebras. The graph $G$ is denoted as \emph{the collapse of $E$ through the set of edges $Y$}.
\end{proposition} 
Given a graph $E$ and a subset $X\subseteq E^0$ ( $Y \subseteq E^1$) we have constructed a particular subgraph $F$ of $E$ that allow us to see a quotient in a path algebra as the path algebra of the graph $F$. 

\begin{definition}\label{def_innertree}
Let $E$ be a directed graph and $A = KE$ the corresponding path algebra over the field $K$. If we consider $X \subset E^0$ the set of vertices that belong to any cycle in $\PP(E)$ we call the \emph{skeleton of $E$}, denoted by $
\Sigma(E) := E \setminus X$, to the collapse of $E$ through $X$.
\end{definition}
 If we consider a set $I$, we define the set of square matrices $\MM_{|I|}(A)$ with entries in $A$ as the square matrices with $|I|$ rows and columns and the elements are of the form $(a_{i,j})$ with $a_{i,j}\in A$ for every $i,j \in I$ and there is only a finite number of nonzero $a_{i,j}$.

If $E$ is a graph with a finite number of vertices, that is, $|E^0| < \infty$, then the path algebra $ KE$ over the field $K$ is unital. Not only that, but $E^0$ is a complete orthogonal system, $1 = \sum_{v \in E^0} v$ and $uv = \delta_{uv} v$. Thus, we can construct the Peirce decomposition of $A = KE$. If we order all the vertices $E^0 = \{v_1,v_2,\ldots,v_n\}$, then $ A = \bigoplus_{i,j = 1}^n A_{i,j}$ being $A_{i,j}:= v_i A v_j$. It is well known that each $A_{i,i}$ are subalgebras of $A$ and $A_{i,j}$ are  $(A_{i,i},A_{j,j})$-bimodules. These structures are compatible with a matrix product and then we can represent 
\begin{equation*}
    A = \left ( \begin{array}{cccc}
        A_{1,1} & A_{1,2} & \cdots & A_{1,n}  \\
        A_{2,1} & A_{2,2} & \cdots & A_{2,n} \\
        \vdots & \vdots & \ddots & \vdots \\
        A_{n,1} & A_{n,2} & \cdots & A_{n,n}
    \end{array}\right ).
\end{equation*}
Each $A_{i,j}$ is the vector space generated by all the paths connecting $v_i$ with $v_j$.  

The set $E^0$ is not the only complete orthogonal set of idempotent elements. If we take a partition of the vertices $E^0 = V_1 \sqcup V_2 \sqcup \cdots \sqcup V_m$ and define $e_1 = \sum_{v \in V_1}v, \cdots, e_m = \sum_{v \in V_m} v$.
As we mention before, this corresponds to a complete orthogonal system of $m$ idempotent elements. In this context, we can construct another Peirce decomposition
\begin{equation*}
      A = \left ( \begin{array}{cccc}
        B_{1,1} & B_{1,2} & \cdots & B_{1,m}  \\
        B_{2,1} & B_{2,2} & \cdots & B_{2,m} \\
        \vdots & \vdots & \ddots & \vdots \\
        B_{m,1} & B_{m,2} & \cdots & B_{m,m}
    \end{array}\right )
\end{equation*}
with $B_{i,j} := e_iAe_j$. This time $B_{i,j}$ is the vector space generated by all the paths that connect one vertex of $V_i$ to a vertex of $V_j$.

One of the applications of the Peirce decomposition of path algebras is the description of the path algebra of a cycle with an arbitrary number of vertices. This will allow us to understand it as a subalgebra of the matrix algebra over the ring of polynomials in one variable. Knowing the behaviour of this structure will be key when we prove the characterisation of the Noetherian path algebras as it will allow us, together with the Peirce decomposition, to study its radical.

\subsection{Cycle with \texorpdfstring{$n$}{n} vertices}\label{ejemplo_ciclo}
For the purpose of this subsection all the subscripts will be modulo $n$. Consider the graph $C_n$
with $C_n^0= \{v_0,v_1,\ldots,v_{n-1}\}$ and $C_n^1 = \{f_0, f_1, \ldots, f_{n-1}\}$ where $s(f_i) = v_i$ for $0 \leq i \leq n-1$, $r(f_i) = v_{i+1}$ for $0 \leq i \leq n-1$. Besides, we define the paths $\mu_{i,i} = v_i$, $\mu_{i,j} = f_if_{i+1} \cdots f_{j-1}$ and $c_i = f_if_{i+1} \cdots f_{i-1}$ for $0\leq i,j \leq n-1$. In this case, if we consider the $K$-algebra $A = KC_n$, we have that $A_{i,j} = K[c_i]\mu_{i,j}$
with $K[c_i]$ the polynomial algebra over the variable $c_i$. This leads us to
\begin{equation*}
    A = \left ( \begin{array}{ccccc}
        K[c_0] & K[c_0]\mu_{0,1} & \cdots & K[c_0]\mu_{0,n-2} &  K[c_0]\mu_{0,n-1} \\
        K[c_1]\mu_{1,0} & K[c_1] & \cdots & K[c_1]\mu_{1,n-2} &K[c_1]\mu_{1,n-1} \\
        K[c_2]\mu_{2,0} & K[c_2]\mu_{2,1} & \cdots & K[c_2]\mu_{2,n-2} & K[c_2]\mu_{2,n-1} \\
        \vdots & \vdots & \ddots & \vdots & \vdots \\
        K[c_{n-1}]\mu_{n-1,0} & K[c_{n-1}]\mu_{n-1,1} & \cdots & K[c_{n-1}]\mu_{n-1,n-2} &  K[c_{n-1}]
    \end{array}\right ).
\end{equation*}

We have the relation $c_if_i=f_ic_{i+1}$. This, immediately implies that
$p(c_i)f_i=f_ip(c_{i+1})$ for any polynomial $p$. And consequently, $\mu_{i,j}p(c_k)=\delta_{j,k}p(c_i)\mu_{i,j}$ where again $p$ is a polynomial and $\delta$ is the Kronecker's Delta.

Notice that for $i>0$ we have $\mu_{0,i}\mu_{i,0}=c_0$ and $\mu_{i,0}\mu_{0,i}=c_i$.
Furthermore:
$$\mu_{i,j}\mu_{j,k}=\begin{cases}\mu_{i,k} & \hbox{ if } i\le j\le k\\
c_i \mu_{i,k} & \hbox{ if } i\le k<j \hbox{ or } j<i\le k\hbox{ or } k<j<i.\\
\end{cases}$$
So, we claim that $\mu_{i,j}\mu_{j,k}=c_i^{n(i,j,k)}\mu_{i,k}$ where $n(i,j,k)=1$ if and only if 
$(i\le k<j)$ or $(j<i\le k)$ or $(k<j<i)$ and $n(i,j,k)=0$ otherwise.

The Peirce decomposition implies that $A = \bigoplus_{i,j=0}^{n-1}K[c_i]\mu_{i,j}$ and a general element is of the form $\sum_{i,j=0}^{n-1} p_{i,j}(c_i)\mu_{i,j}$. The product of two elements is then

\begin{equation*}
\begin{split}
\left (\sum_{i,j=0}^{n-1} p_{i,j}(c_i)\mu_{i,j} \right )\left (\sum_{i,j=0}^{n-1} q_{i,j}(c_i)\mu_{i,j} \right )& =\sum_{i,j,k,l=0}^{n-1}p_{i,j}(c_i)\mu_{i,j}q_{k,l}(c_k)\mu_{k,l}= \\
      = \sum_{i,j,k,l=0}^{n-1}\delta_{j,k}p_{i,j}(c_i)q_{k,l}(c_i)\mu_{i,j}\mu_{k,l} & =\sum_{i,j,l=0}^{n-1}p_{i,j}(c_i)q_{j,l}(c_i)\mu_{i,j}\mu_{j,l}= \\
     & = \sum_{i,j,l=0}^{n-1}p_{i,j}(c_i)q_{j,l}(c_i)c_i^{n(i,j,l)}\mu_{i,l}.
\end{split}
\end{equation*}

Define next the numbers $m(i,j)=1$ if $i>j$ and $m(i,j)=0$ else. Then enable us
to define $\tau\colon KC_n\to \MM_n(K[x])$ such that $\tau \left (\sum_{i,j}p_{i,j}(c_i)\mu_{i,j} \right ):=\sum_{i,j}p_{i,j}(x)x^{m(i,j)}E_{i,j}$ where $E_{i,j}$ is the usual elementary matrix. This map is a monomorphism of $K$-vector spaces. We reproduce here a table with the values of $n(i,j,k)$ and various $m(i,j)$'s for the possible scenarios of the triplet $(i,j,k)$:
\begin{center}
    \begin{tabular}{|c|c|c||c|c|}
    \hline
    & $n(i,j,k)$ & $m(i,k)$ & $m(i,j)$ & $m(j,k)$ \\
    \hline
    $i\le j\le k$ & $0$ & $0$ & $0$ & $0$\\
    $i\le k<j$ & $1$ & $0$ & $0$ & $1$\\
    $j\le i\le k$ & $1$ & $0$ & $1$ & $0$\\
    $j\le k<i$ & $0$ & $1$ & $1$ & $0$\\
$k\le i\le j$ & $0$ & $1$ & $0$ & $1$\\
$k\le j<i$ & $1$ & $1$ & $1$ & $1$\\
    \hline
    \end{tabular}
\end{center}

\noindent consequently, we have $n(i,j,k)+m(i,k)=m(i,j)+m(j,k)$ for any $i,j,k\in\{1,\ldots,n\}$. It can be proved that $\tau$ is a $K$-algebra monomorphism $\tau\colon KC_n\to \MM_n(K[x])$. Thus 
$$K C_n\cong \Imagen(\tau) = \left ( \begin{array}{ccccc}
        K[x] & K[x] & \cdots & K[x]&  K[x] \\
        xK[x] & K[x] & \cdots &K[x]&  K[x] \\
        xK[x] & xK[x] & \cdots & K[x] & K[x] \\
        \vdots & \vdots & \ddots & \vdots & \vdots  \\
        xK[x] & xK[x] & \cdots &  xK[x] & K[x]
    \end{array}\right ).
$$

    \section{Centroid of a path algebra, extended centroid and central closure of  \texorpdfstring{$KC_n$}{KCn}}\label{sec_centroide}

In this section we will focus our attention in the \emph{centroid} of a path algebra. Not only we will study its elements, but also we will describe its structure as an algebra. In relation to this, we will work with the \emph{center} of an algebra. Both concepts are really close and even are the same (in the sense of isomorphisms) when the algebra is unitary. However, when the graph have infinite vertices (that is, the path algebra is not unitary) we can find examples of algebras with trivial center and nontrivial centroid. That is why we are interested in this object, it generalise the center and coincides with it when the algebra is unitary. 
\begin{definition}
Consider a $K$-algebra $A$. The \emph{centroid} of $A$, denoted by $\C(A)$, is the $K$-vector space of all linear maps $T\colon A\to A$ such that $T(xy)=T(x)y=xT(y)$ for any $x,y\in A$. Furthermore, the centroid $\C(A)$ is a $K$-algebra.  The elements of $\C(A)$  will be called \emph{centralizers}.
\end{definition}
If $A$ is prime, it is known that $\C(A)$ is a domain. In addition, if $A$ is a simple algebra, then $\C(A)$ is a field.

Instead of speaking of general $K$-algebras, we can now focus on the path algebras context. If we consider $A = KE$ the path algebra of the graph $E$ over the field $K$ and choose $T \in \C(A)$. In order for this element to be a centralizer, it needs to verify the following equalities:
\begin{equation*}\label{eq_condicion_centralizable}
\left \{
\begin{array}{ll}
      T(u) = T(u)u = uT(u) &  \forall u\in E^0,\\
      0 = uT(v) = T(u)v & \forall u,v \in E^0,\  u\neq v,\\
      T(u)f= T(fv) = fT(v) & \forall f \in E^1 \text{ with } u = s(f), v = r(f).
\end{array} \right .
\end{equation*}
\begin{definition}
Let $E$ be a directed graph and $A = KE$ the corresponding path algebra over the field $K$. We say that an application $T_0 \colon E^0 \to KE$ is \emph{centralisable} if it verifies the conditions given in Equations \eqref{eq_condicion_centralizable}.
\end{definition}

It is not difficult to notice that if $T \in \C(A)$, then $\left .T \right \vert_{E^0}$ is centralisable. Furthermore, the converse is also true. Let $T_0 \colon E^0 \to KE$ be a centralisable application and define the linear application $T\colon KE \to KE$ over the paths as $T(\mu) = T_0(s(\mu))\mu$, then $T \in \C(A)$. In particular, there is a one to one relation between the set of elements of $\C(A)$ and the set of centralisable applications. In other words every centralizer $T \in \C(A)$ with $A$ a path algebra is characterised by its restriction over the set of vertices. 

 If $E$ is a directed graph with $\{E_i\}_{i\in I}$ its connected components we have that $KE = \bigoplus_{i\in I}KE_i$. As a direct consequence, it is true that $  \C(KE) \cong \bigoplus_{i\in I} \C(KE_i)$.
In conclusion, if we characterise the centroid of a path algebra with the corresponding graph $E$ connected, we will be able to describe the centroid of any path algebra known just by studying its connected components. 

\begin{remark}
\label{lema_anulador}
Let $E$ be a connected directed graph and $A = KE$ the path algebra over the field $K$. If we consider an edge $f\in E^1$ with $r(f)=v$ (resp. $s(f)=v$) Then $\text{\rm Rann}_{A}(f)\cap vA=0$ (resp. $\text{\rm Lan}_{A}(f)\cap Av=0$). This can be proven using \cite[Lemma 1]{centro_algebra_caminos}.
\end{remark}
\begin{proposition}\label{corollary_camino_identidad1}
Let $E$ be a connected directed graph and $A = KE$ the path algebra over the field $K$. Let take $u\in E^0$ and $T \in \C(A)$ verifying $T(u) = ku$ for some $k \in K$. If there is a path $\mu$ with $s(\mu) = u$ (resp. $r  (\mu) = u)$ and $r(\mu) = v \in E^0$ (resp. $s(\mu) = v)$, then $T(v)=kv$.
\end{proposition}
\begin{proof}
In order to prove this result, it is only necessary to verify this for edges, the rest follows directly. Also, the dual is completely analogous. Under the previous conditions, taking into account that $T \in \C(A)$, we can verify that $T(f) = T(uf) = T(u)f = kuf = kf = f(kv)$ and $T(f) = T(fv) = fT(v)$. As a consequence $f(kv - T(v)) = 0$, then $kv-T(v) \in {\rm Rann}_A(f)$. It is true that $T(v) = T(vv) = vT(v) \in vA$ and $v = vv \in vA$, then $kv-T(v) \in vA$. Finally $kv - T(v) \in {\rm Rann}_A(f) \cap vA$ and considering Remark \ref{lema_anulador} $kv-T(v) = 0$ and $T(v) = kv$.
\end{proof}
\begin{corollary}\label{theorem_conexo_ku_kId}
Let $E$ be a connected directed graph and $A = KE$ the path algebra over the field $K$. Consider $T \in \C(A)$, if there is some $u \in E^0$ that verifies $T(u) = ku$ for a $k \in K$, then $T = k \rm{Id}$.
\end{corollary}

Before entering in the characterisation of the centroid, we need to recall the concept of \emph{unital associative algebra} generated by a certain set. This will allow us to study the corner $uAu$ in a different way. The following lemmas will be key in the characterisation of the centroid of the path algebra.

\begin{lemma}\label{prop_uAu_Ass}
 Let $E$ be a directed graph. Consider $u \in E^0$ and $X_u \subseteq \PP(E)$ the set of all paths $\lambda = f_1f_2\cdots f_n$ with $s(\lambda) = r(\lambda) = u$ and $r(f_i) \neq u$ for all $i=1,\ldots,n-1$. If $A = KE$ then $uAu \cong {\rm Ass}(X_u)$.
\end{lemma}
\begin{proof}
If we consider the morphism ${\rm Ass}(X) \to uAu$ as $1 \mapsto u$ and $\lambda \mapsto \lambda$. It is obvious that this morphism is surjective. In order to prove that it is injective, we consider any polynomial $p \in {\rm Ass}(X)$ with $\sigma(p) = 0$. In order to prove that $p = 0$, we need to verify that if $q = \lambda_1\lambda_2\cdots \lambda_n$ and $q'=\lambda'_1\lambda'_2\cdots \lambda'_m$ are two monomials such that $\sigma(q) = \sigma(q')$ then $q = q'$. If $\alpha = \lambda_1\lambda_2\cdots \lambda_n  = \lambda'_1\lambda'_2\cdots \lambda'_m = \beta$ (as paths) then these paths go through $u$ the same amount of times. The paths $\alpha $ crosses $n+1$ times, while $\beta$ crosses $m+1$ times, then $n = m$, because $\alpha = \beta$. Suppose that $\el(\lambda_1) \geq \el(\lambda'_1)$. As $\alpha = \beta$, this means that $\lambda_1 = \lambda'_1\gamma$, but $\lambda_1$ goes through $u$ twice while $\lambda'_1\gamma$ three times unless $\gamma = u$ and $\lambda_1 = \lambda'_1$. Then, we have that $\lambda_2\lambda_3\cdots \lambda_n = \lambda'_2\lambda'_3\cdots \lambda'_n$. Repeating this argument we have that $\lambda_i = \lambda'_i$ for $i=1,\ldots,n$ and $q = q'$. If $p = \sum_{i=1}^n a_iq_i$ with $q_i$ monomials in ${\rm Ass}(X)$ with $q_i \neq q_j$ and $a_i \neq 0$, we have that $\sigma(p) = \sum_{i=1}a_i \sigma(q_i)$ with $\sigma(q_i) \in \PP(E)$ and $\sigma(q_i) \neq \sigma(q_j)$ by the previous result. Finally, $\sigma(p) = 0$ if and only if $a_i = 0$ for every $i=1,\ldots,n$ and this implies that $p = 0$.
\end{proof}

\begin{lemma} \label{lemma_centro_Ass}
If $X$ is a set with $|X| > 1$, then $\Z(\A_K(X)) = K1$.
\end{lemma}
\begin{proof}
In order to prove this result we define the centralizer of an element $z \in \A_K(X)$ as the set ${\rm C}(z) := \{a \in \A_K(X) \vert \ az = za  \}$. It is not difficult to see that $\Z(\A_K(X)) = \bigcap_{z \in \A_K(X)} {\rm C}(z)$. Following Theorem \cite[5.3]{centralizadores} we can affirm that ${\rm C}(z) = K[z]$ for every $z \in \A_K(X)$. As a consequence, $\Z(\A_K(X)) = \bigcap_{z \in \A_K(X)}K[z] = K1$.
\end{proof}
\begin{lemma}\label{lemma_centralizador_en_centro}
Let $E$ be a connected directed graph and $A = KE$ the path algebra over the field $K$. Consider $T: A \to A$ with $T \in \C(A)$. If $u \in E^0$, then $T(u) \in \Z(uAu)$.
\end{lemma}
\begin{proof}
Under the previous conditions, in order to prove $T(u) \in \Z(uAu)$ it suffices to show that $T(u) \in uAu$ and $T(u)a = aT(u)$ for every $a \in uAu$, which is immediate under the definition of an element of the centroid.
\end{proof}

Now we are in conditions to characterise the centroid of a path algebra in case that the graph is connected. We recall that we focus in this setting because the centroid commutes with the direct sum of the connected components of the graph. In \cite{quivlecs} W. Crawley gives a characterisation of the centroid for a finite graph. We will proof this result for arbitrary graphs.

\begin{theorem}\label{teo_centroide}
Let us consider $E$ a connected directed graph and $A = KE$ the corresponding path algebra over a field $K$. If $E$ is a cycle, then $\C(A) \cong K[x]$, in other case $\C(A) = K\rm{Id}_{A}$.
\end{theorem}
\begin{proof}
If we consider the sets $E^0$ and $X_u$ as in Lemma \ref{prop_uAu_Ass} for every $u \in E^0$, there are two possibilities. 
There is a vertex $u \in E^0$ such as $|X_u| \neq 1$. In this case we have another two possibilities.

If $|X_u| = 0$, that is, $X_u = \emptyset$ we have by Lemma \ref{prop_uAu_Ass} that  $uAu \cong K$. In particular, $uAu = Ku$. Consider $T \in \C(A)$, by Lemma \ref{lemma_centralizador_en_centro}, $T(u) \in \Z(uAu) = \Z(Ku) = Ku$. Then, $T(u) = ku$. Corollary \ref{theorem_conexo_ku_kId} implies that $T = k\rm{Id}$ and $\C(A) = K\rm{Id}$ as a consequence. 

In case that $|X_u|> 1$, we know from the Lemma \ref{lemma_centro_Ass} that $\Z(\A_K(X_u)) = K1$. In the same way as before, if $T \in \C(A)$, then $T(u) \in \Z(uAu) = Ku$ and $T(u) = ku$. We apply Corollary \ref{theorem_conexo_ku_kId} once more and we have that $T = k\rm{Id}$. This leads us to $\C(A) = K\rm{Id}$. 

Suppose that for all $u \in E^0$ we have that $|X_u| = 1$. In this case $uAu \cong \A_K(X_u) \cong K[x]$. In particular, $uAu = K[c_u]$ being $c_u$ the only closed path that starts and ends in $u$. In this case we have two more possibilities. 

Suppose that $E$ is not a cycle. Consider $u \in E^0$ and $c_u$ the closed path that starts in $u$ and ends in $u$. Without loss of generality, we can suppose that there is an edge $f$ that connects $u$ with  $v\not\in c_u^0$. Furthermore, there is another cycle $c_v$ that starts and ends in $v$. Otherwise, it would not be true that $|X_u| = 1$ for every $u \in E^0$.

If $c_u^0\cap c_v^0 \neq \emptyset$ it would not be true that $|X_u| = 1$ for every $u \in E^0$ because $w \in c_u^0 \cap c_v^0$ would verify $|X_w| \geq 2$. As a consequence $c_u^0\cap c_v^0 = \emptyset$. Without loss of generality we can suppose $s(f) = u$ y $r(f) = v$. Consider $T \in \C(A)$. From Lemma \ref{lemma_centralizador_en_centro} we have that $T(u) \in \Z(uAu) = K[c_u]$ and $T(u) = k_0 u + \sum_{i = 1}^n k_i c_u^i$ with $k_i \in K$. On the other hand, we have that $T(v) = l_0v + \sum_{j = 1}^m l_j c_v^j$. If we are familiar with the properties of the elements of $\C(A)$ we have that $T(f) = T(uf) = T(u)f = k_0 f + \sum_{i=1}^n k_i c_u^if$ and $T(f) = T(fv) = fT(v) = l_0 f + \sum_{j=1}^ml_i fc_v^j$ with $l_j \in K$. Thus
\begin{equation*}
            (k_0-l_0)f + \sum_{i=1}^nk_ic_u^if - \sum_{j = 1}^m l_j fc_v^j = 0.
\end{equation*}

From previous hypothesis we have that all the paths implied in this last linear combination are different, and so $k_0 = l_0 = k$, $k_i = 0$ for every $i = 1, \ldots,n$ and $l_j = 0$ for all $j = 1,\ldots,m$. This allows us to prove that $T(u) = ku$ and using again Corollary \ref{theorem_conexo_ku_kId}, we have that $T = k \rm{Id}$. This give us $\C(A) = K\rm{Id}$. 

The last possibility is that $E$ is a cycle $c$ with a finite number of vertices and this implies that $\C(A) \cong \Z(A)$. Following \cite[Theorem 1]{centro_algebra_caminos} we have that $\Z(A) \cong K[x]$, then $\C(A) \cong K[c]$.
\end{proof}

 \subsection{Extended centroid and central clousure of \texorpdfstring{$KC_n$}{KCn}}\label{susec_centroideextendido}
 
 Since the centroid of the path algebra $KC_n$, with $C_n$ a cycle with $n$ vertices, is not trivial, one step forward would be to study its \emph{extended centroid} and its \emph{central closure}. As we already know from Subsection \ref{ejemplo_ciclo}, we can consider $KC_n$ as a subalgebra of $\MM_n(K[x])$ which is really useful because this means that $KC_n$ is a polynomial identity algebra thanks to Amitsur-Levitzki Theorem in \cite[Theorem 1]{amitsur_levitzki}. In addition, \cite[Theorem 4.1]{PI_algebras} establishes that the extended centroid of PI-algebras is isomorphic to the field of fractions of its center. In this case, we know that $\Z(KC_n) \cong K[x]$ (\cite[Proposition 3]{centro_algebra_caminos}) and the extended centroid is $\Gamma(KC_n) \cong K(x)$. However, knowing the extended centroid will not be enough. For the aim of using the construction of the central closure given in \cite{baxter_central_closure} we will need the explicit isomorphism $\Gamma(KC_n) \to K(x)$.

 The algebra $KC_n$ is a prime algebra, as we will prove in Proposition \ref{prop_caracterizacion_prima}, then we are in conditions to make use of the construction of the extended centroid and central closure given in \cite{baxter_central_closure}. First, let us consider the pair $(f,I)$ with $I$ a nontrivial ideal of $KC_n$ and $f\colon I \to KC_n$ a morphism of $KC_n$-bimodules. This pair will be called admissible and we can define a equivalence relation in the set of admissible pairs. We define it as $(f,I) \sim (g,J) \iff \left . f \right \vert_{I\cap J} =  \left . g \right \vert_{I\cap J}$.
 We define the $K$-algebra given by the set $\Gamma(KC_n) = \{(\overline{f,I}) \colon (f,I) \text{ is an admissible pair}\}$ with the operations
 \begin{equation*}
 \begin{split}
      k \cdot (\overline{f,I}) & := (\overline{kf,I}), \\
      (\overline{f,I})+ (\overline{g,J}) & := (\overline{f+g, I\cap J}), \\
       (\overline{f,I}) \cdot (\overline{g,J}) & := (\overline{f\circ g, g^{-1}(I)}).
 \end{split}
 \end{equation*}
   Denote $A = KC_n$, if we consider an admissible pair $(f,I)$, we have that $I_{i,i} := I \cap A_{i,i}$ is an ideal of $A_{i,i} = K[c_i]$. Then, $I_{i,i} = (q_i(c_i))$, the ideal generated by $q_i(c_i) \in K[c_i]$. Consider $q (x) = \lcm \{q_i(x)\}  \in K[x]$ we define the ideal $J := \bigoplus_{i,j=1}^n J_{i,j}$ with $J_{i,j} := (q(c_i))\mu_{i,j}$. The ideal $J$ is the maximum nontrivial ideal contained in $I$ such that $J \cap A_{i,i} = (p(c_i))$ for $p(x) \in K[x]$. It is obvious that $(f,I) \sim (f,J)$. We have that $f(q(c_i)) = f(v_iq(c_i)v_i) = v_if(q(c_i))v_i \in v_iAv_i = K[c_i]$. This means that $f(q(c_i)) = p_i(c_i)$. On one hand, we have that $f(q(c_i)\mu_{i,j})= f(q(c_i))\mu_{i,j}=p_i(c_i) \mu_{i,j}$, whereas, $f(q(c_i)\mu_{i,j}) = f(\mu_{i,j}q(c_j)) = \mu_{i,j}f(q(c_j)) = \mu_{i,j}p_j(c_j)= p_j(c_i)\mu_{i,j}$.  Thus, we have that $p_i(c_i) = p_j(c_i) = p(c_i)$ for every $0\leq i,j \leq n-1$ and as a consequence $f(q(c_i))= p(c_i)$ for every $i = 0,\ldots, n-1$. This will allow us to consider for an admissible pair $(f,I)$ a quotient of polynomials $\frac{p(x)}{q(x)}$.  
 \begin{proposition}\label{prop_centroide_extendido_ciclo}
 If we consider $A = KC_n$ the path $K$-algebra of a cycle of $n$ indexed vertices over $\{0,\ldots, n-1\}$. The correspondence $\Omega \colon \Gamma(A) \to K(x)$ given by
 \begin{equation*}
    (\overline{f,I}) \mapsto \Omega\left ( (\overline{f,I}) \right ) :=\frac{p(x)}{q(x)}
 \end{equation*}
 is an isomorphism, in other words, $\Gamma(A) \cong K(x)$.
 \end{proposition}
 \begin{proof}
 In order to prove that the application is well defined we need to verify that the image of $(\overline{f,I})$ is the same independently of the element we use of the equivalence class. If we consider $(\left . f \right \vert_{I'},I')$ with $I' \subseteq I$ we have that their corresponding $J' \subseteq I' \subseteq I$ and $J'$ an ideal of $KC_n$, but $J$ is the maximum ideal verifying this property. Thus, $J'\subseteq J$ and this means, in particular, that $J'_{i,i} = (q'(c_i)) \subseteq (q(c_i)) = J_{i,i}$. This implies that $q'(x) = r(x)q(x)$. And $f(q'(c_i)) = f(r(c_i)q(c_i)) = r(c_i)p(c_i)$. Finally, this admissible pair is in correspondence with $\frac{r(x)p(x)}{r(x)q(x)} = \frac{p(x)}{q(x)}$. 
 For the second part, we need to prove that this application is an morphism of $K$-algebras. The fact that $\Omega$ behaves well with the scalar multiplication is straightforward.
 Let us take $(\overline{f,I}), (\overline{g,I'}) \in \Gamma(A)$ with $f(q(c_i)) = p(c_i)$ and $g(q'(c_i)) = p'(c_i)$. If we consider $q_0(x) = \lcm\{q(x),q'(x)\}$ and $J_0$ an ideal with $(J_0)_{i,j} = (q_0(c_i))\mu_{i,j}$. It is not difficult to check that, by construction, $J_0 \subseteq J \subseteq I$ and $J_0 \subseteq J' \subseteq I'$, thus, $J_0 \subseteq I\cap I'$. Denoting $Q(x) = \gcd\{q(x),q'(x)\}$ we have that $(f+g)(q_0(c_i)) = f(q_0(c_i)) +g(q_0(c_i)) = f(r(c_i)q(c_i))+g(r'(c_i)q'(c_i)) = r(c_i)p(c_i)+r'(c_i)p'(c_i)$ with $r(x) = \frac{q'(x)}{Q(x)}$ and $r'(x) = \frac{q(x)}{Q(x)}$ (because $q_0(x)Q(x) = q(x)q'(x)$). If we do some computation it is easy to verify that
 \begin{equation*}
 \begin{split}
      \Omega((\overline{f,I})+(\overline{g,I'})) & = \Omega((\overline{f+g,J_0})) =  
     \Omega((\overline{f,I}))+ \Omega((\overline{g,I'})).
 \end{split}
 \end{equation*}
 Now, if we consider $L$ a two-sided ideal with $L_{i,j} = (q(c_i)q'(c_i))\mu_{i,j}$ we have that $L \subseteq g^{-1}(I)$ because $g(q(c_i)q'(c_i)) = q(c_i)g(q'(c_i)) = q(c_i)p'(c_i) \in J \subseteq I$. Then $f(g(q(c_i)q'(c_i))) = f(q(c_i)p'(c_i)) = f(q(c_i))p'(c_i) = p(c_i)p'(c_i)$ and 
 \begin{equation*}
     \Omega((\overline{f,I})\cdot(\overline{g,I'})) = \Omega((\overline{f\circ g,L})) = \frac{p(x)p'(x)}{q(x)q'(x)} = \Omega((\overline{f,I})) \Omega((\overline{g,I'})).
 \end{equation*}
 Finally, it is obvious that this morphism is surjective. For an element $\frac{p(x)}{q(x)}$, we just need to consider the element $(\overline{f,J})$ with $J$ an ideal verifying $J_{i,j} = (q(c_i))\mu_{i,j}$ and $f \colon J \to A$ defined as $f(q(c_i)) = p(c_i)$. If we consider $(\overline{f,I})$ with $\Omega((\overline{f,I})) = 0$ this means that there is an ideal $J\subseteq I$ with $J_{i,j} = (q(c_i))\mu_{i,j}$ and $f(q(c_i)) = 0$ for $i = 0,\ldots, n-1$. This means that $f(J) = 0$ and $(\overline{f,I}) = (\overline{0,J}) = 0$.
 \end{proof}

 Now that we know how the isomorphism $\Omega \colon \Gamma(A) \to K(x)$ behaves, we are in conditions to construct the \emph{central closure} of $A$.  An element $x \in A\otimes \Gamma(A)$ is vanishing if there is a representation $\sum_{k=1}^m a_k \otimes \lambda_k$ with $a_k \in A$ and $\lambda_k = (\overline{f_k,I})$ for every $k = 1,\ldots,m$, verifying $\sum_{k = 1}^m P(a_k)f_k(z) = 0$
 for every $z \in I$ and $P \in \M(A)$ where $\M(A)$ is the algebra generated by the identity and the morphisms $L_a$ and $R_b$ with $a,b\in A$. Since $A$ is an associative algebra, it is easy to see that $\M(A) = \left \{\sum_{a,b} L_aR_b \colon a,b \in A \right \}$ (the sums are finite). Following \cite{baxter_central_closure}, the central closure of $A$ is $\widehat{A} :=  \frac{A\otimes \Gamma(A)}{M}$ with $M$ the set formed by vanishing elements in $A\otimes \Gamma(A)$. The set $M$ is an ideal of $A \otimes \Gamma(A)$ as a $\Gamma(A)$-algebra (\cite[Lemma 2.6]{baxter_central_closure}) and we can form the quotient. 
 Even though we have a construction of the central closure of $A$, Proposition \ref{prop_centroide_extendido_ciclo} will let us characterise the central closure of $A$ as a better understanding algebra. 
 
 \begin{theorem}\label{teorema_clausura_central_ciclo}
 If we consider the path algebra $A = KC_n$, then the central closure of $A$, that is, $\widehat{A}$ is isomorphic to $\MM_n(K(x))$. 
 \end{theorem}
  
 \begin{proof}
 Let us consider the morphism of $\Gamma(A)$-algebras $\theta \colon A \otimes \Gamma(A) \to \MM_n(K(x))$ such that $ \theta\left ( \sum_{k=1}^m a_k \otimes \lambda_k \right ) := \sum_{k=1}^m \tau(a_k)\Omega(\lambda_k)$, with $\tau$ the isomorphism given in Subsection \ref{ejemplo_ciclo} and $\Omega$ the isomorphism given in Proposition \ref{prop_centroide_extendido_ciclo}. It will only be necessary to prove that $\ker(\theta) = M$ and that $\theta$ is surjective. 
  
  Let us consider a vanishing element, that is, $\sum_{k=1}^m a_k \otimes \lambda_k$ with $\lambda_k = (\overline{f_k,I})$ and $\sum_{k=1}^m P(a_k)f_k(z) = 0$ for all $P \in \M(A)$ and $z \in I$. If we consider the ideal $J \subset I$ with $J_{i,j} = (q(c_i))\mu_{i,j}$ and $f_k(q(c_i)) = p_k(c_i)$, we have that $  \sum_{k=1}^m a_kf_k\left (\sum_{i=1}^nq(c_i)  \right ) =  \sum_{k=1}^m a_k\sum_{i=1}^n p_k(c_i)  = 0$.
  If we apply the morphism $\tau \colon A \to \MM_n(K[x])$ to this element, we have that $\tau\left (\sum_{k=1}^m a_k\left( \sum_{i=1}^n p_k(c_i) \right ) \right ) =  \sum_{k=1}^m \tau(a_k)\tau\left( \sum_{i=1}^n p_k(c_i) \right )= \sum_{k = 1}^m \tau(a_k)p_k(x) = 0$. In consequence, $\theta \left (\sum_{k=1}^m a_k \otimes \lambda_k \right ) =\sum_{k = 1}^m \tau(a_k)\Omega(\lambda_k) = \sum_{k = 1}^m \tau(a_k)\frac{p_k(x)}{q(x)} = 0$. This proves that $M \subseteq \ker \theta$. 
  
  In order to prove the other contention we will make use of the maximality of $M$ given in Lemma \cite[2.11]{baxter_central_closure}. This Lemma establish that $M$ is the unique ideal of $A\otimes \Gamma(A)$ maximal with  respect to the property that $I \subseteq M$ and $M \cap (A\otimes 1) = 0$. Here, $I$ denotes the ideal generated by some particular vanishing elements. Then, the only thing we need to verify is that $\ker \theta \cap (A \otimes 1) = 0$. If $a \otimes 1 \in \ker \theta$ we have that $\theta(a\otimes 1) = \tau(a)\Omega(1) = \tau(a) = 0$ and because of $\tau$ is injective it is true that $a = 0$ and $a\otimes 1 = 0$.
  
  Finally, we need to prove that this morsphism is surjective. This is not difficult to obtain since all the matrix in $\MM_n(K(x))$ can be written as $\frac{1}{q(x)}B$ with $B \in \tau(KC_n)$. Then this matrix is the image of the element $a \otimes \lambda$, with $a = \sum_{i,j = 1}^n r_{i,j}(c_i)\mu_{i,j}$ where $r_{i,j} \in K[x]$ and $\lambda = (\overline{f,I})$ with $I_{i,j} = ((q(c_i))\mu_{i,j}$ and $f(q(c_i)) = v_i$ for $1\leq i,j \leq n$.
 \end{proof}

Describing the central closure of the path algebra of a cycle as an algebra of matrices, with coefficients in the field of fractions of polynomial, instead of as a quotient of a tensor product, eases the computations and the understanding of the behaviour of this algebra. Also, working with algebras of matrices over a field will allow to use, under certain conditions, well known theorems such as Wedderburn-Artin Theorem. 
We can say that the central closure and the extended centroid are good definitions because they keep being isomorphic for isomorphic algebras. Checking this fact is not difficult but it needs of some computations. In addition, the extended centroid and the central closure commute with direct sums. Again, this is not difficult but it can result tedious and far from the main objective in this work.

\section{Artinian and semiartinian path algebras}\label{sec_artinian}
The purpose of this section is the characterisation of Artinian and semiartinian path algebras. As we will prove, this will depend on several conditions of the own graph, like the number of vertices and edges or the existence of closed paths.

The stationary descending chain condition of an Artinian algebra can be easier studied if we know the behaviour of the involved ideals. That is the reason why being able to describe certain ideals of the path algebras in terms of their generators will make the description of Artinian path algebras easier

\begin{proposition}\label{prop_caracterizacion_artiniano}
Let $E$ be a graph and $A = KE$ the corresponding path algebra over the field $K$. The following statements are equivalent
\begin{enumerate}
    \item The algebra $A$ is left (right) Artinian.
    \item $E$ does not contain closed paths and $|E^0 \cup E^1| < \infty $.
    \item The algebra $A$ is finite dimensional. 
\end{enumerate}
\end{proposition}
\begin{proof}
We will only prove the equivalence between the first and the second statement. The equivalence between the second and third statement is well known in the path algebra theory.

If $|E^0 \cup E^1| < \infty $ and $E$ does not contain closed paths, then Proposition \ref{prop_caracterizacion_artiniano} implies that $A$ is finite dimensional and it is well known that this means left (right) Artinian. 
Reciprocally, let us consider that $A$ is left Artinian, the proof is analogous for right Artinian. If $|E^0\cup E^1| = \infty$, then $|E^0| = \infty$ or $|E^1| = \infty$. Let us suppose that $|E^0| = \infty$ (the other case is completely analogous). We can consider the descending chain of ideals.
\begin{equation*}
    \sum_{u \in E^0} Au \supset \sum_{u \in E^0\setminus \{v_1\}} Au \supset \sum_{u \in E^0\setminus\{v_1,v_2\}} Au \supset \cdots \supset \sum_{u \in E^0\setminus \{v_1,\ldots,v_n\}} Au \supset \cdots 
\end{equation*}
As we will see, this chain is strictly descending. Let us suppose that $$\sum_{u \in E^0\setminus \{v_1,\ldots,v_n\}} Au= \sum_{u \in E^0\setminus \{v_1,\ldots,v_n,v_{n+1}\}} Au.$$ In particular, this means that $v_{n+1} \in \sum_{u \in E^0\setminus \{v_1,\ldots,v_n,v_{n+1}\}} Au$ and that is not possible. 
On the other hand, if we consider that there is a closed path $\gamma \in A$. We construct the strictly descending chain of ideals
\begin{equation*}
    A\gamma \supset A\gamma^2 \supset \cdots \supset A\gamma^n \supset \cdots .
\end{equation*}
Thanks to the same arguments that we used before, we can observe that this is a strictly descending chain (it is not stationary), and that can not be possible. 
\end{proof}
\begin{remark}
As a consequence of Proposition \ref{prop_caracterizacion_artiniano}, in the path algebras context, left Artinian algebra and right Artinian algebra are equivalent. In fact, it can be proven (in the same way as Proposition \ref{prop_caracterizacion_artiniano}) that it is enough to verify the condition of descending chain for two-sided ideals. 
\end{remark}

In relation to Artinian algebras we have, in a less restricted way, semiartinian algebras. These algebras are described in terms of the socle.

\begin{definition}
    Let us consider a $K$-algebra $A$ and $M$ a left (resp. right) $A$-module. We define \emph{the left (resp. right) socle of $M$}, denoted by $\soc(M)$ as the sum of all the minimal left (resp. right) submodules of $M$. In the path algebras context, we define $\soc_l(A) := \soc (_AA)$ (resp. $\soc_r(A):= \soc(A_A)$), that is, the sum of all the minimal left (resp. right) ideals of $A$.
\end{definition}

\begin{definition}
    We say that the $K$-algebra $A$ is \textit{left (resp. right) semiartinian} if for every proper left (resp. right) ideal $I$ we have that $\soc(A/I) \neq \{0\}$.
\end{definition}
The definition of semiartinian algebras makes use of the socle of a module. Then, we will need to know the behaviour of the minimal ideals. As we will acknowledge in the following proposition, they are related to the sources and the sinks. When we try to build a minimal left (resp. right) ideal, the first we think of is to construct a one dimensional ideal. This ideal must take the form of $I = Kx$ with $x$ a certain element in $KE$. If we do not choose wisely our generator of the ideal, there might be an element $y \in KE$ such that $yx \neq kx$ (resp. $xy = 0$) and $I$ is not one dimensional. However, if we consider $x$ a source (resp. a sink) there is no possible path (or element) in $y \in KE$ that makes $yx \neq kx$. This means that sources and sinks will be related to the minimal ideals that we can construct in $KE$. We will illustrate this with more details in the next proposition.
\begin{proposition}\label{prop_zocalo_pathalgebra}
    Let us consider $E$ a directed graph. The left ideal $I$ is minimal if and only if $I = K x$ with $x = \sum_{i=1}^n k_i \lambda_i$ ($k_i \in K^{\times}$, $\lambda_i \in \PP(E)$ and $\lambda_i \neq \lambda_j$ if $i\ne j$) and $s(\lambda_i) = v$ with $v \in \source(E)$, for every $i=1,\ldots,n$. In addition, if $x = \sum_{i=1}^n k_i \lambda_i$ and $x' = \sum_{j=1}^m l_i \mu_j$, we have $Ix \cong Ix'$ as a left modules if and only if $s(\lambda_i) = s(\mu_j)$ for every $1\leq i,j \leq n,m$. 
\end{proposition}
\begin{proof}
    Let us consider $A = KE$ and $I$ a minimal left ideal. We can consider $I = Ax$ with $x \in A$ because  $I$ is minimal. If we write $x = \sum_{i=1}^n k_i \lambda_i$ with $k_i \in K^{\times}$ and $\lambda_i \in \PP(E)$ with $\lambda_i \neq \lambda_j$ for $i\neq j$, there is a vertex $v \in E^0$ with $vx \in Ax \setminus \{0\}$, then $I = Avx$ and without loss of generality we can consider $x = vx$, that is, all the paths starts in $v$. If $v$ is not a source this means that there is an edge $f \in E^1$ such that $fv = f$. We can consider then $fx \in Ax \setminus \{0\}$. Because $I$ is minimal, we have that $I = Afx = Ax$, this implies that there is $y \in A$ such that $yfx = x$. Then we have that $  \sum_{i,j=1}^{n,m} l_jk_i\mu_jf\lambda_i = \sum_{i=1}^n k_i \lambda_i$ with $l_j \in K^{\times}$ and $\mu_j \in \PP(E)$. This implies that $\sum_{i,j=1}^{n,m} l_jk_i\mu_jf\lambda_i - \sum_{i=1}^n k_i \lambda_i = 0$ which is a linear combination of different paths which can not be possible. Thus, $v$ must be a source and $Av = Kv$. Finally, $I = Ax = Avx = Kv x = K x$. 
    
    Observe that the reverse is trivial because the ideal $I = Kx$ is one dimensional. 

    Let us now consider $\varphi \colon Kx \to Kx'$ an isomorphism of left modules (left ideals). If we have that $v = s(\lambda_i)$ for $i = 1,\ldots,n$ then $kx' = \varphi(x) = \varphi(vx) = v\varphi(x) = kvx' $  which implies that $v = s(\mu_j)$. For the converse, we just have to consider the isomorphism $\theta \colon Kx \to Kx'$ such that $\theta(x) = x'$. 
\end{proof}
    \begin{remark}\rm
    The dual of this result is also true. That is, if $I$ is a nonempty minimal right ideal of $KE$,  then $I = Kx$ with $x$ a linear combination of paths that end in the same sink. In addition, if $x = \sum_{i=1}^n k_i \lambda_i$ and $x'=\sum_{j=1}^m l_j \mu_j$ we have that $Kx \cong Kx'$ if and only if $r(\lambda_i) = r(\mu_j) = v \in \sink(E)$ for every $1\leq i,j \leq n,m$.  
\end{remark}
Once we know how the minimal ideals in $KE$ behave we are in conditions to describe the left and right socles. We have that $\soc_l(KE) = \sum Kx$ with $x$ such as Proposition \ref{prop_zocalo_pathalgebra}. If continue computing we have that $\soc_l(KE) = \sum Kx = \bigoplus_{s(\lambda) \in \source(E)} K\lambda $ with $\lambda \in \PP(E)$. If we group each of the elements in the sums by the isomorphic class, we have that $vA = \bigoplus_{s(\lambda) = v} K\lambda$ with $ v \in \source(E)$. Thus, $\soc_l(KE) = \bigoplus_{v \in \source(E)}vA$. 
\begin{corollary}\label{corolario_zocalo_pathalgebra}
    Let us consider $E$ a directed graph. Then left socle and right socle of $A = KE$ are respectively $\soc_l(A) = \bigoplus_{v \in \source(E)} vA$ and $ \soc_r(A) = \bigoplus_{v \in \sink(E)} Av$. That is, left socle is the ideal generated by $\source(E)$ and the right socle is the ideal generated by $\sink(E)$ (both of them two-sided ideals). 
\end{corollary}
The sum of all isomorphic (as $KE$-modules) minimal left ideals is a two-sided ideal that (following the classical theory) we might call \textit{homogeneous component}.

As we know, the left socle of an algebra $A$ is a two-sided ideal. If we consider the quotient $A/\soc_l(A)$ we obtain a new algebra and we can try to compute again the left socle of it. Once again, $\soc_l(A/\soc_l(A))$ is a two-sided ideal in $A/\soc_l(A)$ and the ideals in $A/\soc_l(A)$ come from ideals in $A$ that contain $\soc_l(A)$. That is, $\soc_l(A/\soc_l(A)) = S/\soc_l(A)$ with $S$ an ideal in $A$ with $\soc_l(A) \subseteq S$. We can continue this process of computing left socles of different quotient until we obtain an ascending chain of ideals. In this sense we can define the following.
\begin{definition}\label{def_chainsocle}
    Let $A$ be a $K$-algebra. We define the \emph{left (resp. right) ascending socle chain of ideals} as the well-ordered chain of two sided ideals
    \begin{equation*}
        0 \subset S_1 \subset S_2 \subset \cdots \subset S_\alpha \subset \cdots 
    \end{equation*}
    with $S_{\alpha+1}/S_\alpha = \soc_l(A/S_\alpha)$ (resp. $S_{\alpha+1}/S_\alpha = \soc_r(A/S_\alpha)$) and $S_\omega = \cup_{\alpha < \omega} S_\alpha$ if $\omega$ is a limit ordinal.
\end{definition}

In the path algebras context, Corollary \ref{corolario_zocalo_pathalgebra} let us compute the left and right chain of socles. If we consider a directed graph $E$ and its corresponding path algebra $KE$, we know that $\soc_l(KE)$ is the ideal generated by $\source(E)$. As we know, the quotient $KE/\soc_l(KE)$ is isomorphic to the path algebra $K(E\setminus \source(E))$ and if we want to compute the socle of this new path algebra we need the sources of the graph $E\setminus \source(E)$. The reader might not find difficult to prove that, following the notation in Definition \ref{def_chainsocle}, $S_2$ is the ideal generated by the vertices $\source(E) \cup \source(E \setminus \source(E))$. We can keep this reasoning and obtain the whole ascending socle chain. Although all of these constructions are given in terms of the left socles and the sources of a graph, it is possible to replicate all this process in terms of the right socles and the sinks. In fact, during the rest of this section, we will focus on the left socles but all the results provided here also hold for the right socles.

In order to ease the notation we establish the next definition.
\begin{definition}\label{def_ifuente}
    Consider a directed graph $E$, we define $\f(E) := \source(E)$. In general, $\f^0(E) := \emptyset$, $\f^1(E) := \f(E)$ and $\f^i(E) := \f \left (E \setminus \bigcup_{j < i} \f^j(E) \right )$ for every ordinal $i >0$. We denote \emph{the set of $i$-sources} by $\f^i(E) \subset E^0$ .
\end{definition}
{
With this new notation it is not difficult to check that $S_\alpha$ is the ideal generated by $\bigcup_{i \leq \alpha} \f^i(E)$. Besides, it will ease the characterisation of the semiartinian path algebras. As we will check, left semiartinian path algebras are highly connected to the $i$-sources. We first need these two key results.
}

\begin{lemma}\label{prop_conexion_isources}
      If  $u \in \f^i(E)$ with $i > 1$ there is an edge $e \in E^1$ such that $r(e) = u$ and $s(e) \in \f^j(E)$ with $j<i$. In addition, if $i$ is not a limit ordinal, we can consider $j = i-1$.
\end{lemma}
  \begin{proof}
      Let us consider $u \in \f^i(E)$ with $i > 1$, and denote by $F_i =  E \setminus \cup_{j < i}\f^j (E)$. If $r(e) \neq u$ for every $e \in E^1$ then $u \in \source(E) = \f^1(E)$ and this enter in contradiction with our hypothesis. Thus, there is an edge $e \in r^{-1}(u)$. Furthermore, we have that $ e \notin F_i^1$, if not $u \notin \f \left ( F_i\right ) = \f^i(E)$. This leads to $s(e) \notin F_i^0$ which implies that $s(e) \in \cup_{j < i}\f^j (E)$. 

      Let us also consider that $i$ is not a limit ordinal and suppose that $s(e) \notin \f^{i-1}(E)$ for every $e \in r^{-1}(u)$. Thus, we have that $\{s(e) \colon r(e) = u\} \subseteq \cup_{j<i-1}\f^j(E)$ and $u \in \f\left ( E \setminus\cup_{j< i-1}\f^j(E)\right ) = \f^{i-1}(E)$ which is not true.
  \end{proof}

\begin{proposition}
\label{lema_semiartinian}
    Let $E$ be a directed graph and $KE$ its corresponding path algebra. Consider a left ideal $I$ such that $\soc_l(KE/I) = \{0\}$, then $\{\lambda \in \PP(E) \colon s(\lambda) \in \f^i(E)\} \subset I$ for every ordinal $i$. In particular, $\f^i(E) \subset I$ for every ordinal $i$. 
\end{proposition} 
\begin{proof}
    We will approach this proof by transfinite induction in the order of the sources. Consider $A = KE$ and take a path $\lambda$ with $s(\lambda) \in f^1(E)$ and consider $I\cap A\lambda$. We have two possibilities
\begin{itemize}
    \item If $I\cap A\lambda = \{0\}$ then $(I+A\lambda) / I \cong A\lambda / \{0\} \cong K\lambda$ and this means that $\soc_l(A/I) \neq \{0\}$ which is not possible. 
    \item If $I \cap A\lambda \neq \{0\}$ then $\lambda \in I$ by the minimality of $A\lambda$.
\end{itemize}
Consider that $I$ contains all the all the paths that start in a $i$-source. Let us take a path $\lambda$ such that $s(\lambda) \in \f^{i+1}(E)$ and consider $I \cap A\lambda$. 
\begin{itemize}
    \item If $I\cap A\lambda = \{0\}$ then $(I+A\lambda) /I \cong A\lambda$ and $\soc_l(A/I) \neq \{0\}$ and this is not possible.
    \item If $I \cap A\lambda \neq \{0\}$ then we have two possibilities. If there is $x \in I \cap A\lambda$ with $x = \sum_{j=1}^n k_j\lambda_j\lambda + k\lambda$ with $\lambda_j \neq \lambda_{j'}$, $k,k_j \in K$ and $k \neq 0$, then $\lambda = k^{-1} (x- \sum_{j=1} k_j \lambda_j\lambda) \in I$ because $\lambda_j\lambda \in I$ for every $j = 1,\ldots,n$ by our induction hypothesis. If there is no such element, this means that $K\lambda \cong A\lambda / (I\cap A\lambda) \cong (I+As(\lambda)) /I $ and $\soc_l(A/I) \neq \{0\}$ and this means that option is not possible. 
\end{itemize}
Finally, for $v \in \f^\alpha(E)$ with $\alpha$ a limit ordinal, the proof is completely analogous because Lemma \ref{prop_conexion_isources} is for sources of every order.
\end{proof}
Now we have the tools to characterise the semiartinianity of a path algebra in terms of geometric conditions of its graph. 
\begin{theorem}\label{prop_semiartiniana}
    Let $E$ be a directed graph and $A = KE$ its corresponding path algebra over the field $K$. The algebra $A$ is left semiartinian if and only if for every $u \in E^0$ there is an ordinal number $\alpha$ such that $u \in \f^\alpha (E)$.
\end{theorem}
\begin{proof}
    Consider that $KE$ is left semiartinian and suppose that there is a set of vertices $X \subset E^0$ such that for every $u \in X$ there is not any ordinal number $\alpha$ such that $u \in \f^\alpha (E)$. If we consider $I$ the ideal generated by the vertices $Y = E^0 \setminus X$ we have that $ KE/I \cong K (E\setminus Y)$. There are not any sources in the graph $E \setminus Y$ and by Corollary \ref{corolario_zocalo_pathalgebra} we have that $\soc_l(KE/I) \cong \soc_l(K(E\setminus Y)) = \{0\}$ but this is only possible if $X = \emptyset$.

    For the converse, consider a left ideal $I$ such that $\soc_l(KE/I) = \{0\}$. By Proposition \ref{lema_semiartinian} the ideal $I$ contains all the $\alpha$-sources, in particular, $E^0 \subset I$ by our hypothesis which implies that $I = KE$. Proving that $KE$ is left semiartinian. 
\end{proof}
\begin{remark}\rm
If the algebra $KE$ is left semiartinian, we will see that this implies that the graph $E$ has no cycles: if we consider the ideal $I$ generated by all the vertices that do not belong to a cycle, denoted by $X$, we have that $KE/I \cong K(E\setminus X)$ has a trivial socle since $E\setminus X$ is a graph without sources (each connected component in the graph is strongly connected). The only possible way for this to happen is that $I = KE$, that is, $X = E^0$. In addition, this result, together with Theorem \ref{prop_semiartiniana}, implies that the vertices that are basis of a cycle are not sources of any order. Besides, if the graph is finite, the algebra $KE$ is not only left semiartinian, Proposition \ref{prop_caracterizacion_artiniano} implies that $KE$ is a finite dimensional, thus, we have that $E$ is a tree.
\end{remark}
\begin{corollary}
    Consider $E$ a finite directed graph and $A = KE$ the corresponding path algebra over the field $K$. The following assertions are equivalent:
    \begin{enumerate}
        \item The algebra $A$ is finite dimensional.
        \item The algebra $A$ is left Artinian.
        \item The algebra $A$ is left semiartinian.
        \item The graph $E$ is acyclic. 
    \end{enumerate}
\end{corollary}

\begin{example}
Consider the directed graph $L$ which set of vertices is $L^0 = \{v\} \cup \{u_{i,j} \colon i \in \mathbb{N} \setminus \{0\}, 1\leq j \leq i\}$ and set of edges is $L^1 = \{f_{i,j} \colon i \in \mathbb{N}\setminus \{0\}, 1 \leq j \leq i\}$ with $s(f_{i,j}) = u_{i,j}$ for every pair $(i,j)$, $r(f_{i,j}) =u_{i,j+1}$ if $j < i$ and $r(f_{i,i}) = v$ for every $i \in \mathbb{N}$. We denote this as the Levi graph. One can think of the Levi graph as an countable infinite union of increasing finite lines connected through the same vertex $v$. The reader can find a graphic representation of the Levi graph in Figure \ref{fig_levigraph}. For a finite ordinal $\alpha$ we have that $\f^\alpha(L) = \{u_{i,\alpha} \colon i \geq \alpha\}$ and for the limit ordinal $\omega = \mathbb{N}$ we have that $\f^\omega (L) = \{v\}$. In particular, $\f^\omega(L) \cup \left ( \bigcup_{\alpha =1}^\infty \f^\alpha(L)\right ) = L^0$, that is, every vertex is a $i$-source and the path algebra $KL$ is semiartinian.

    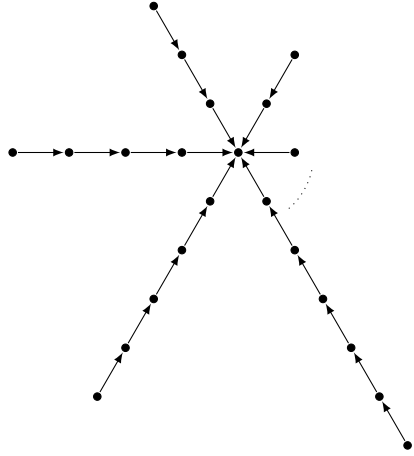
\begin{figure}[ht]
    \centering
    \begin{tikzpicture}[shorten <=2pt,shorten >=2pt,>=latex, node distance={15mm}, main/.style = {draw, fill, circle, inner sep = 1pt}]

\node[main] (0) at (0,0) {};
\def \n {6}
\def \margin {8} 
\foreach \radius in {1,...,\n}
{
\foreach \s in {\radius,...,\n}
{
  \node[main] at ({360/\n * (\s - 1)}:0.75*\radius) {};
  \draw[->] ({360/\n * (\s - 1)}: 0.75*\radius) to ({360/\n * (\s - 1)}:{0.75*\radius-0.75});
}
   }
  \draw[dotted] ({360/\n * (\n - 1)+\margin}:1) 
    arc ({360/\n * (\n - 1)+\margin}:{360/\n * (\n)-\margin}:1);

\end{tikzpicture}
    \caption{Levi graph. }
    \label{fig_levigraph}
\end{figure}
\end{example}

 In general, being acyclic is not enough to imply that the path algebra is semiartinian (if the graph is not finite).  In Figure \ref{fig_lineainfinita}, we have an example of a directed graph which its corresponding path algebra is not semiartinian.
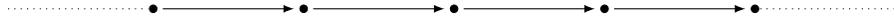
\begin{figure}[ht]
    \centering
    \begin{tikzpicture}
[shorten <=2pt,shorten >=2pt,>=latex, node distance={15mm}, main/.style = {draw, fill, circle, inner sep = 1pt}]

\foreach \i in {1,...,5}
{
\node[main] ({\i}) at ({2 * \i},0) {};
}
\draw[->] (1) to (2);
\draw[->] (2) to (3);
\draw[->] (3) to (4);
\draw[->] (4) to (5);
\draw[dotted] (10,0) to (12,0);
\draw[dotted] (0,0) to (2,0);
    \end{tikzpicture}
    \caption{Directed graph $E$ such that there is no sources and no cycles.}
    \label{fig_lineainfinita}
\end{figure}
It is very important to highlight the fact that all these previous result have their corresponding dual. Instead of talking about left semiartinian, right socle and sources of order $i$, we refer to right semiartinian, right socle and sinks of order $i$. All of the proofs above are analogous for the dual results just by taking the proper considerations.

If the directed graph $E$ and $KE$ is left semiartinian, then we can consider a total ordering in the set $E^0$ such that the adjacency matrix $A_E$ is upper triangular. We just need to order each set $\f^\alpha(E)$ (Zermelo's Theorem) and consider if $u \in \f^i(E)$ and $v \in \f^j(E)$ such that $i< j$, then $u<v$ and if $u,v \in \f^i(E)$ we just consider the previous established order. Following this order of the vertices and by Lemma \ref{prop_conexion_isources} the adjacency matrix $A_E$ is upper triangular.

\section{Prime, semiprime and primitive path algebras}\label{sec_prime}
Having characterized certain {finiteness} conditions on $KE$ (Artinian and semiartinian) directly from its graph, we pursue now to characterise {perfection} algebraic properties of $KE$, again, through conditions on their corresponding graphs. 
The first condition that one could analyze when studying a given algebraic structure is its semiprime character. 

\begin{definition}
An associative algebra is said {\it semiprime} if and only if for every $a \in A$ we have that $aAa = 0$ implies that $a = 0$, in other words, if $a \neq 0$, then $aAa \neq 0$.
An associative algebra is said {\it prime} if and only if for every $a,b \in A$ with $aAb = 0$, then $a = 0$ or $b = 0$, in other words, if $a \neq 0$ and $b \neq 0$, then $aAb \neq 0$.
\end{definition}
These conditions are well known to be equivalent to 
\begin{enumerate}
\item For any ideal of $A$, $I^2=0$ implies $I=0$ (semiprimeness)
\item For any two ideals $I$ and $J$ of $A$, $IJ=0$ implies $I=0$ o $J=0$ (primeness).
\end{enumerate}
\begin{remark}\label{lemma_semiprimo}
In case of the algebra $A$ is graded, there is a new characterisation of primeness and semiprimeness in terms of the homogeneous components. Consider $A = \bigoplus_{m \in \mathbb{Z}}A_m$ a $\mathbb{Z}$-graded associative algebra. Then $A$ is semiprime if and only if $a \in A_m$ with $aAa =0$ implies $a = 0$ for every $m \in \mathbb{Z}$. 
Consider $A = \bigoplus_{m \in \mathbb{Z}}A_m$ a $\mathbb{Z}$-graded associative algebra. Then $A$ is prime if and only if $a \in A_n,b \in A_m$ with $aAb =0$ implies $a = 0$ or $b = 0$ for every $n,m \in \mathbb{Z}$.
\end{remark}

As we have already mentioned, the next result is proven by M. Siles in \cite{prop_semiprimo_mercedes} using Leavitt path algebras. In this case, we prove it just making use of the path algebra.
\begin{proposition}\label{prop_caracterizacion_semiprima}

 Let $E$ be a directed graph and $A = KE$ the corresponding path algebra over the field $K$. We have that $A$ is semiprime if and only if for every path $\mu$ in $E$ there is another path $\lambda$ with $  s(\lambda) = r(\mu) \ \text{and} \ r(\lambda) = s(\mu)$. In other words, if there is a path $\mu$ from $u$ to $v$, with $u,v \in E^0$, then there is another path $\lambda$ from $v$ to $u$.
\end{proposition}
\begin{proof}
First, consider $A$ semiprime and take any non trivial path $\mu$. By definition, this path verifies that $\mu A\mu \neq 0$. This is only possible if there is a path $\lambda$ with $\mu \lambda \mu \neq 0$, and then, $s(\lambda) = r(\mu)$ and $r(\lambda) = s(\mu)$.

On the other hand, consider that for every path $\mu$ from $u$ to $v$ there is another path $\lambda$ from $v$ to $u$ and let us take a non trivial element $a \in A_m$. We can consider $a = \sum_{i = 1}^n k_i \mu_i$ with $k_i \in K\setminus \{0\}$ and all $\mu_i$ paths of length $m$ with $\mu_i \neq \mu_j$ for $i\neq j$. Choose a certain path $\mu_l$ with $1\leq l\leq n$. As our hypothesis says, there is a path $\lambda$ with $s(\lambda) = r(\mu_l)$ and $r(\lambda)= s(\mu_l)$. If we compute $a\lambda a$ we have that
\begin{equation*}
    a \lambda a= \left ( \sum_{i = 1}^n k_i \mu_i\right ) \lambda \left (\sum_{j = 1}^n k_j \mu_j\right ) =  \sum_{i,j = 1}^n k_ik_j \mu_i \lambda \mu_j = \sum_{(i,j) \in J} k_ik_j \mu_i \lambda \mu_j,
\end{equation*}
where $J = \{(i,j) \colon r(\mu_i) = r(\mu_l), s(\mu_j) = s(\mu_l)\}$. The set $J \neq \emptyset$ because $(l,l) \in J$. Also, $k_ik_j \neq 0$ and all the paths $\mu_i \lambda \mu_j$ are different because all the paths $\mu_i$ are different but with the same length.
We have then a non trivial linear combination of different paths (linearly independents) and so $a\lambda a = \sum_{(i,j) \in J} k_ik_j \mu_i \lambda \mu_j \neq 0$. This implies that $aAa \neq 0$ and $A$ is semiprime following Lemma \ref{lemma_semiprimo}. 
\end{proof}

Now let us focus on primeness of $KE$. Again, there is an approach (in a specific ambient) in \cite[Theorem 3.5]{prop_semiprimo_mercedes}. On the other hand, 
the reference \cite{quivlecs} mentions, but not proves, that primeness of $KE$ is equivalent to the strongly connectedness of the graph $E$.
Also P. Smith in \cite{paulsmith}  proves the equivalence of primeness of $KE$ with the strongly connectedness of $E$. Though the ambient graph in the aforementioned references is supposed to be finite, the statement in \cite[Lemma 2.11]{paulsmith} is also true for arbitrary graphs. This reference is not easily accessible and as far as we know has not been published in a regular journal or book. That is why we consider suitable to include an ample proof of the result.
\begin{proposition}\label{prop_caracterizacion_prima}
Let $E$ be a directed graph and $A = KE$ the corresponding path algebra over the field $K$. We have that $A$ is prime if and only if for every pair of vertices $u,v \in E^0$ there is a path $\lambda$ with $s(\lambda) = u$ and $r(\lambda) = v$. That is $KE$ is prime if and only if $E$ is strongly connected.
\end{proposition}
\begin{proof}
Let us consider that $A$ is a prime algebra and take $u,v \in E^0$, then we have $uAv \neq 0$ and necessarily there is a path $\mu$ with $u\lambda v \neq 0$. As a consequence, $s(\lambda) = u$ and $r(\lambda) = v$.

On the other hand, take into account that for every pair of vertices there is a path that connect them. Consider $a \in A_n$ and $b \in A_m$ with $a \neq 0$, $b \neq 0$ and $a = \sum_{i=1}^k r_i \mu_i$ and $b = \sum_{j = 1}^l s_j \nu_j$ with $r_i,s_j \in K\setminus \{0\}$ for all $i= 1,\ldots,k$ and $j= 1,\ldots,l$. Also, all the paths in $a$ are different but with the same length $n$ and all the paths in $b$ are different with the same length $m$. Take two paths $\mu_{i_0}$ and $\nu_{j_0}$ in the respective linear combinations. Consider $u = r(\mu_{i_0})$ and $v = s(\nu_{j_0})$. As our hypothesis says, there is a path $\lambda$ with $s(\lambda) = u$ and $r(\lambda) = v$. Then
\begin{equation*}
    a\lambda b = \left ( \sum_{i = 1}^l r_i \mu_i\right ) \lambda \left (\sum_{j = 1}^ks_j \nu_j\right ) =  \sum_{i,j = 1}^{l,k} r_is_j \mu_i \lambda  \nu_j = \sum_{(i,j) \in J} r_is_j \mu_i \lambda \nu_j.
\end{equation*}
where $J = \{(i,j) \colon r(\mu_i) = s(\lambda), s(\nu_j) = r(\lambda)\}$. We have that $J \neq \emptyset$ because $(i_0,j_0) \in J$. Besides, all the paths $\mu_i\lambda \nu_j$ are different for the different pairs $(i,j)$ because $\mu_i \neq \mu_{i'}$ if $i \neq i'$ (but with the same length) and $\nu_j \neq \nu_{j'}$ if $j \neq j'$ (but with the same lenght too). In consequence, we have a nontrivial linear combination of different paths, then $a \lambda b \neq 0$ and $aAb \neq 0$. Following Remark \ref{lemma_semiprimo}, this means that $A$ is prime.  
\end{proof}

Primeness and semiprimenss of algebras are related,  prime implies semiprime, but they are also linked to simpleness and primitiveness of algebras. As we know, simple implies primitive and primitive implies prime. In the path algebras context, they are even more related. As we will see, the basic blocks of what semiprime path algebras are constructed are: simple, primitive and prime algebras. To be more precise we will see that any semiprime algebra is a direct sum of simple, path algebra of cycle 
and primitive ones.

Let $E$ be a connected directed graph and $A = KE$ the corresponding path algebra over the field $K$. If we consider that $E$ is formed by just one vertex, that is, $E^0 = \{u\}$ and $E^1 = \emptyset$, it is easy to see that $KE = Ku \cong K$ and then, $A$ is simple. On the converse, consider $A$ being a simple algebra. Let us take $u \in E^0$ and define the two-sided ideal $I = AuA$. Under our hypothesis, this means that $I = 0$ or $I = A$. It is clear that $I \neq 0$ because $Ku \subset I$. If there is another vertex $v \in E^0$ with $v \neq u$, we have that $v \notin I$ and $I \neq A$. Finally, $E^0 = \{u\}$. In this way, we have that the algebra $A =KE$ is simple if and only if the graph $E$ is just a vertex.

Once we have characterised prime and semiprime path algebras, the next logic step is to study primitive path algebras. The associative context of the path algebras allow us to use two different characterisations of primitive algebras. One in terms of representations and the other one in terms of ideals. 

\begin{definition}
    An associative ring $R$ is said \emph{left (resp. right) primitive} if there is a faithful simple left (resp. right) $R$-module.
\end{definition}
 In the Leavitt path algebras context, the primitive algebras have already been characterised. This is stated in Theorem \cite[4.4.10]{lpa}. The Leavitt path algebra $L_K(E)$ is primitive if and only if the graph $E$ is downward directed and satisfies Condition $(L)$. The graph $E$ is \emph{downward directed} if we consider two vertices $u,v \in E^0$ there is a third vertex $w \in E^0$ and two paths $\lambda,\mu \in \PP(E)$ with $s(\lambda) = u$, $s(\mu) = v$ and $r(\lambda) = r(\mu) = w$. This is also equivalent to $L_K(E)$ being a prime algebra. On the other hand, condition $(L)$ means that every cycle has an exit. In the path algebras context, characterisation of primitive algebras is amusingly similar to the Leavitt path algebra case, but adapting condition $(i)$ to the equivalent in path algebras. 
 
It is also important to highlight that the study of primitive Leavitt path algebras is for row finite directed graphs, but for the path algebras context we do not need any conditions of finiteness.

\begin{theorem}\label{teorema_primitiva}
    Let $E$ be a directed graph and $K$ any field. The path algebra $KE$ is left primitive if and only if
    \begin{enumerate}
        \item[(i)] $E$ is strongly connected, and
        \item[(ii)] $E$ satisfies Condition $(L)$.
    \end{enumerate}
    \begin{proof}
        Let $E$ be a row-finite directed graph and $K$ any field. Consider that the path algebra $A=KE$ is left primitive, in particular, this implies that $A$ is prime \cite[Propostion 11.6]{lam1991first}. Following Proposition \ref{prop_caracterizacion_prima}, the graph $E$ must be strongly connected. On the other hand, suppose that $A$ is a primitive algebra and there is a cycle without exit, this mean that there is a vertex $u \in E^0$ with the corner $uAu \cong K[x]$ which is not a left primitive algebra, this enters in contradiction with $A$ being left primitive because all the corners must be left primitive.

        For the converse, let us suppose that $E$ satisfies conditions $(i)$ and $(ii)$. Condition $(i)$ implies that $A$ is prime algebra and this allow us to use the result in \cite[1.2]{anquela1998local}. If we prove that there is an idempotent element $e \in A$ with $eAe$ primitive, then $A$ is primitive. Consider a vertex $u \in E^0$, we have that $uAu \cong \A(X)$ with $X$ the set of closed paths $\lambda \in \PP(E)$ with $s(\lambda) = r(\lambda) = u$ and go through $u$ just these two times (see Lemma \ref{prop_uAu_Ass}). Conditions $(i)$ and $(ii)$ imply that $|X| \geq 2$ and this implies that $uAu$ is primitive. This last result was proved by E. Formanek and appears in \cite[p. 185]{lam1991first} before Theorem \cite[11.27]{lam1991first}.
    \end{proof} 
\end{theorem}
 
 In general, left primitivity and right primitivity are not equivalent. There are examples of left primitive algebras that are not right primitive \cite{bergman1964ring}. In the Leavitt path algebras context, the existence of an involution makes left primitivy and right primitivity notions equivalent. Despite the fact that in the path algebras context we do not have an involution, the symmetry in the conditions of Theorem \ref{teorema_primitiva} makes left and right primitivity equivalent. That is to say, the dual condition in Theorem \ref{teorema_primitiva} would be that every cycle has an entry. However, if the graph $E$ is strongly connected, every cycle having and entry is equivalent to every cycle having an exit, which makes both notions equivalent.

Let us consider a directed graph $E$ such that its path algebra $A = KE$ is semiprime. We can consider each connected component in the graph $E$ and describe $A$ as a directed sum of path algebras. That is, if $\{E_i\}_{i\in I}$ are the connected components of $E$, then $A = \bigoplus_{i\in I} A_i$ with $A_i = KE_i$. If $A$ is semiprime, Proposition \ref{prop_caracterizacion_semiprima} implies that each $E_i$ is strongly connected, that is, each $A_i$ is a prime algebra (Proposition \ref{prop_caracterizacion_prima}). In this case, the algebra $A_i$ can be simple, primitive or prime. If $E_i$ is just a vertex, then the algebra is simple; if $E_i$ has just one cycle, then $E_i = C_n$ and $A_i$ is prime; if $E_i$ has more than just one cycle, then every cycle has exit and $A_i$ is a primitive algebra. In conclusion, we have the following theorem.

\begin{theorem}[\bf Structure Theorem for semiprime path algebras]\label{teorema_estructura_semiprima}  Let $E$ be a directed graph and $K$ a field. Every semiprime path algebra $KE$ can be decomposed as
\begin{equation*}
    KE = \bigoplus_{i \in I} Ku_i \oplus \bigoplus_{j\in J} KC_{n_j} \oplus \bigoplus_{l \in L} KE_l
\end{equation*}
with $I$ the set of indices such that $E_i^0 = \{u_i\}$ for every $i \in I$, $J$ the set of indices such that $E_j = C_{n_j}$ for every $j \in J$ and $L$ the set of indices such that $KE_l$ is primitive.\end{theorem}

\section{Noetherian path algebras}\label{sec_noetherian}

Developing a condition for a path algebra to be Artinian has no been difficult. We have been able to work with the descending chain condition successfully. In the following, we will work with the ascending chain condition, that is, we will develop a characterisation of Noetherian path algebras. This time it will be not as easy, but the results are promising. Left Noetherianity of path algebras is mentioned in \cite{quivlecs} to be equivalent to the graph being finite and every cycle having no entries. However, the proof of this result is not given. In order to prove so we first need to see that the algebra $KC_n$ is left (resp. right) Noetherian. This is immediate following the Peirce decomposition that we described in Example \ref{ejemplo_ciclo} and Proposition \cite[1.7]{goodear}. Then, we have the following Lemma.

\begin{lemma}\label{lemma_ciclo_noetheriano}
The path algebra $KC_n$  of a cycle with $n$ vertices over a field $K$ is left (right) Noetherian. 
\end{lemma}

Lemma \ref{lemma_ciclo_noetheriano} establishes that a cycle with $n$ vertices, as in Example \ref{ejemplo_ciclo}, is left and right Noetherian. The next question is: what happens when we add exits or entries to the cycle? The intuition tells us that it will preserve the Noetherianity for one side and not for the other.

\begin{proposition} \label{prop_noetheriano_unciclo}
Consider a directed graph $E$ with only one cycle which has no entries (resp. exits) and $|E^0 \cup E^1|< \infty$. Then, the path algebra $A = KE$, with $K$ a field, is left (resp. right) Noetherian.
\end{proposition}
\begin{proof}
We will prove this result for the left Noetherian case. The other one is completely analogous (or dual). The arguments of the proof will rely on several results of \cite{goodear} and the Peirce decomposition of $KE$.

Consider $c$ the cycle without entries and $F^0 = E^0\setminus c^0$. Without any loss of generality we can order the vertices such that $u_i \in c^0$ for $i=1,\ldots,n$ and $u_i \in F^0$ for $i=n+1,\ldots,n+m$. Let us take the complete ortonormal system of idempotent elements $e_1 = \sum_{i=1}^nu_i $ and $e_2 = \sum_{i=n+1}^{n+m}u_i$, that is, $e_1e_2 = e_2e_1=0$, $e_j^2=e_j$ for $j =1,2$ and $e_1+e_2 =1$. As we know $c$ has no entries, thus, there is no path from a vertex in $F^0$ to a vertex in $c^0$ and $A_{21}:= e_2Ae_1 = 0$. In consequence, we have that
\begin{equation*}
    A \cong \left ( \begin{array}{cc}
        S & B  \\
       0  & T
    \end{array} \right )
\end{equation*}
with $S := e_1Ae_1$, $T:= e_2Ae_2$ and $B := e_1Ae_2$.
Following \cite[Proposition 1.8]{goodear}, if we prove that $S$ and $T$ are left Noetherian and $B$ is a finitely generated left $S$-module, we will prove that $A$ is left Noetherian. As we know from the Peirce decomposition, $S$ and $T$ are subalgebras of $A$ and $B$ is a $(S,T)$-bimodule. The algebra $S$ is isomorphic to the path algebra of a cycle and following Lemma \ref{lemma_ciclo_noetheriano}, $S$ is Noetherian. If we consider the algebra $T$, we have that it is finite dimensional because there is a finite number of paths connecting vertices in  $F^0$ (there are no cycles and the number of vertices and edges is finite). As a consequence, $T$ is left Noetherian.

For the last part we need to prove that $B$ is finite generated as a left $S$-module. If we consider $G$ the subgraph of $E$ obtained by removing the cycle except the vertices that have exits. It is not difficult to check that the graph $G$, under our hypothesis, does not contain any cycle and $|G^0\cup G^1|< \infty$. This means that the number of paths in $G$ is finite. Let us now prove that $B = \sum_{\mu \in \PP(G)}S \mu$. As we know, $B$ is the vector space generated by all the paths that start at some vertex in $c^0$ and end in a vertex in $F^0$. From that, it is easy to verify that $B' = \sum_{\mu \in \PP(G)}S \mu  \subseteq B$. On the other hand, if we consider a path $\lambda \in B$ we have that $\lambda = \beta \mu$ with $\beta$ a path in $Kc$ and $\mu$ a path in $KG$. Finally, for all the paths in $B$ we have that $\lambda = \beta \mu \subseteq S \mu \subseteq B'$ and $B \subseteq B'$. This proves that $B$ is finitely generated as a $S$-module and concludes the proof. 
\end{proof}

Now that we know some examples of Noetherian path algebras, and we are aware of how the quotient of path algebras over two-sided ideals, generated by vertices, behave. We are in conditions to prove the next characterisation.

\begin{proposition}\label{prop_noetheriano} 
Let $E$ be a directed graph and $K$ a field. The path algebra $KE$ is left (resp. right) Noetherian  if and only if $|E^0\cup E^1| <  \infty $ and every cycle has no entries (resp. exits). 
\end{proposition}
\begin{proof}
We will prove the left Noetherian case. The right one is completely analogous. In fact, it is the dual result.

First, suppose that $A = KE$ is left Noetherian. Consider that $E^0\cup E^1$ is not finite. This means that $E^0$ or $E^1$ are infinte sets. If $E^0$ is not finite there is a set of vertices $\{u_i\}_{i \geq 1}$ and then we can construct the chain of left ideals
\begin{equation*}
    Au_1 \subset Au_1+Au_2 \subset \cdots \subset  \sum_{i=1}^nAu_i \subset \cdots 
\end{equation*}
It is easy to verify that $u_{n+1} \notin \sum_{i=1}^n Au_i$. Then, we have constructed a strictly ascending chain of left ideals. This can not be possible because $KE$ is Noetherian. As a consequence, $|E^0|<\infty$.

In the same way we can prove that $|E^1|<\infty$. This leads us to $|E^0\cup E^1|<\infty$.

Let us suppose that there is a cycle $c \in \PP(E)$ with an entry $f \in E^1$, that is, $0 \neq fc \in \PP(E)$. In the same way as we did before, we construct the chain of left ideals
\begin{equation*}
    Afc \subset Afc + Afc^2 \subset \cdots \subset \sum_{i=1}^n Afc^i \subset \cdots 
\end{equation*}
This is a strictly ascending chain of left ideals and that enters in contradiction with the fact that $KE$ is left Noetherian. As a consequence, every cycle has no entries. 

Let us prove the reciprocal implication. Consider $E$ a graph with $|E^0\cup E^1| < \infty$ and every cycle has no entries. We will prove that $KE$ is left Noetherian by induction on the number of cycles.

If $E$ has no cycles, we have that $KE$ is a finite dimensional $K$-algebra and this meas that $KE$ is left Noetherian. 

The case $n = 1$ is just a consequence of Proposition \ref{prop_noetheriano_unciclo}.

Suppose as our induction hypothesis that if the graph has $n$ cycles then $KE$ is left Noetherian. Consider a finite graph $E$ with $n+1$ cycles with no entries. Let us take $c$ any cycle and consider $I = (c^0)$ the two-sided ideal generated by the vertices of $c^0$. Following Proposition \cite[1.2]{goodear}, if we prove that $I$ and $KE/I$ are $KE$-Noetherian, then $KE$ will be Noetherian.

As we know from Proposition \ref{prop_colapso} it is true that $KE/I \cong KF$ where $F = E \setminus c^0$ (the collapse of $E$ through the set of vertices $c^0$). Because all the cycles in $E$ have no entries, the graph $F$ has $n$ cycles with no entries, and by our induction hypothesis, this implies that $KF$ is $KF$-Noetherian. In addition, all the $KE$-left modules in $KE/I$ are $KE/I$-left modules in $KE/I$ and viceversa. This is a consequence of the surjectivity of the canonical projection. Then, we have that $KE/I$ is $KE$-Noetherian. 

On the other hand, if we consider the subgraph $G = (G^0,G^1,s,r)$ with 
\begin{equation*}
    G^0 = \{u \in E^0 \colon \exists \ \mu \in I \cap \PP(E) \text{ with } u \in \mu^{(0)}\}
\end{equation*}
\begin{equation*}
     G^1 = \{f \in E^1 \colon \exists \ \mu \in I \cap \PP(E) \text{ with } f \in \mu^{(1)}\}.
\end{equation*}
we can think of this subgraph as the graph formed by selecting all the vertices and edges in the paths that generate $I$. It is important to notice some things.
\begin{enumerate}
    \item The graph $G$ has only one cycle. This fact is because every cycle in $E$ has no entries. As a consequence, there is no path $\mu \in I$ of the form $\mu = c\lambda c'$, with $c'$ another cycle. In other words, there is no path connecting to cycles. 
    \item The set $I$ is contained in $KG$. As we know, the ideal $I$ is generated by all the paths that goes through some vertex $u \in c^0$ and all of them are in $\PP(G)$ by construction. Then $I\subset KG$. Finally, because $G$ is a subgraph of $E$, $KG$ is a subalgebra of $KE$ implying that $I$ is an ideal of $KG$.
    \item If $ \nu \in \PP(E)\setminus\PP(G)$ then $\nu \cdot(X) = 0$. This implies that all the left $KG$-submodules contained in $I$ are left $KE$-modules. Let us take $J$ a left $KG$-module contained in $I$. Consider $a \in KE$, we can write $a = \sum_{i=1}^n a_i \mu_i + \sum_{j = 1}^m b_j \lambda_i $ with $a_i,b_j \in K\setminus \{0\}$, $\mu_i \in \PP(G)$ with $\mu_i \neq \mu_l$ for $i \neq l$ and $\lambda_j \in \PP(E)\setminus \PP(G)$ with $\lambda_j \neq \lambda_k$ for $j\neq k$. Then $aJ =(\sum_{i=1}^n a_i \mu_i + \sum_{j = 1}^m b_j \lambda_i )J = (\sum_{i=1}^n a_i \mu_i)J \subset J  $ and $J$ is a left $KE$-module. In addition, all the left $KE$-modules contained in $I$ are left $KG$-modules because $KG$ is a subalgebra of $KE$.
\end{enumerate}
Finally, thanks to the case $n=1$ we have that $KG$ is left Noetherian and this implies that $I$ is left Noetherian as a $KG$-module. But, as we know, all the left $KE$-submodules in $I$ are left $KG$-submodules in $I$ and viceversa which means that $I$ is left Noetherian.
\end{proof}

\begin{proposition}\label{prop_noeth_formal_triangular}
Let us consider a directed graph $E$ and $A = KE$ the corresponding path algebra over the field $K$. If the algebra $A$ is left or right Noetherian, then $A$ is isomorphic to a triangular matrix algebra. That is, $A$ is isomorphic to 
\begin{equation*}
    \left ( \begin{array}{cc}
        S & B\\
        0 & T
    \end{array}\right ) 
\end{equation*}
with $S$ and $T$ two $K$-algebras and $B$ a $(S,T)$-bimodule.
\end{proposition}

\begin{proof}
As we know from Proposition \ref{prop_noetheriano} being left (resp. right) Noetherian means that $|E^0\cup E^1|< \infty $ and all the cycles in the graph have no entries (resp. exits). Considering these facts, if we define $S^0 = \{v \in E^0 \colon v \text{ is in a cycle} \}$ (resp. $S^0 = \{v \in E^0 \colon v \text{ is not in a cycle}\}$) and $T^0 = E^0 \setminus S^0$. These two sets will allow us to construct a complete system of ortogonal idempotents $ e_1 = \sum_{v \in S^0} v , \ \ e_2 = \sum_{v \in T^0} v$.
For the last part we just need to consider the Peirce decomposition with $S = e_1Ae_1$, $B= e_1Ae_2$ and $T = e_2Ae_2$. It is not difficult to check that under our hypothesis $e_2Ae_1 = 0$.
\end{proof}
\begin{observation}
Under the previous hypothesis $T$ is isomorphic to $K\Sigma(E)$ if $KE$ is left Noetherian and $S$ is isomorphic to $K\Sigma(E)$ if $KE$ is right Noetherian. 
\end{observation}
\begin{observation}
The reader may not find very difficult to check that if $A$ is left (resp. right) Noetherian the algebra $S$ (resp. $T$) is of the form
\begin{equation*}
    \left (\begin{array}{cccc}
        S_1 & 0 & \cdots & 0  \\
         0& S_2 & \cdots & 0 \\
         \vdots & \vdots & \ddots & \vdots \\
         0 & 0 & \cdots & S_m
    \end{array} \right ) 
\end{equation*}
with $m$ the number of cycles and $S_i$ the algebra of a cycle for every $i = 1,\ldots, m$. 
\end{observation}
Once we have a characterisation of the Noetherian path algebras, we would like to obtain an structure theorem for them. If we have isomorphic path algebras, then this result would give us a one-to-one relation between the geometric properties of the corresponding graphs. The \emph{Jacobson radical} of an associative algebra, together with the central clousure of the path algebra of a cycle, will be key to achieve our aim. The central clousure is discussed in Section \ref{sec_centroide} and the characterisation of the Jacobson radical of a path algebra is mentioned (without a proof) in \cite{quivlecs} for finite acyclic graphs. However, we will prove the same result for arbitrary graphs.

Consider a ring $R$ not necessarily unital  and define the product  $a\circ b=a+b-ab$. An element $z \in R$  es said to be  quasi-regular if there exists $z'\in R$ such that $z\circ z'=z'\circ z=0$ (see \cite[p. 8]{jacobson1956structure}).  Then, in Proposition \cite[1, p. 9]{jacobson1956structure} it is proved that 
the radical is the set of all $z\in R$ such that $azb$ is quasi-regular for any $a,b\in R$.

\begin{proposition}\label{prop_radical_path}
    Consider a directed graph $E$ and the path algebra $KE$ over the field $K$. Then the radical $\rad (KE)$ is the vector space generated by the set of paths without return
    \begin{equation*}
        \{\lambda \in \PP(E) \colon \len(\lambda) \geq 1, \nexists \ \mu \in \PP(E), s(\mu) = r(\lambda), r(\mu) = s(\lambda)\}.
    \end{equation*}
\end{proposition}
\begin{proof}
Consider $\lambda$ a path without return. We have to prove that $a\lambda b$ is quasi-regular for any $a,b\in KE$. We have $(a\lambda b)\circ(-a\lambda b)=a\lambda ba\lambda b$ and if we write $ba=\sum_i k_i\tau_i$ with $k_i\in K$ and $\tau_i\in\PP(E)$, we have $\lambda ba\lambda=\sum_i k_i \lambda\tau_i\lambda=0$ since $\lambda$ is a path without return. Thus $a\lambda b$ is quasi-regular implying $\lambda\in\rad(KE)$. Consequently, $\rad(KE)$ contains the linear span of all no-return paths.

 Conversely, assume that a linear combination of paths is in $\rad(KE)$. Those paths in the linear combination without return are in $\rad(KE)$ so assume that we have in $\rad(KE)$ a linear combination of paths all of them with return. Let $z:=\sum_i k_i\lambda_i$ be a such combination. Since $z$ is a finite sum $z = \sum_{u,v} uzv$
where $u,v\in E^0$, there is no loss of generality assuming that $z:=\sum_i k_i\lambda_i\in\rad(KE)$ where $s(\lambda_i)=u$ and $r(\lambda_i)=v$ for every $i$. Furthermore, given that the $\lambda_i$ are paths with return we may assume even that 
$z:=\sum_i k_i\lambda_i\in\rad(KE)$ with $s(\lambda_i)=r(\lambda_i)=u$ for each $i$. But then, taking into account Proposition \cite[1, p. 48]{jacobson1956structure}, $z\in uKEu\cap\rad(KE)=\rad(uKEu)=0$, because $uKEu$ is a free associative algebra (Lemma \ref{prop_uAu_Ass}) and it is left primitive (\cite[Proposition 11.23]{lam1991first}) hence its radical is zero (even if the number of generators of the free algebra is less than or equal to $1$, because in this case the algebra is $K$ or $K[x]$ whose radical is zero).
\end{proof}

Although we already have obtained a characterisation of the Jacobson radical of a path algebra in terms of its generators as a vector space, we can go deeper and describe it in terms of its generators as an ideal. This is shown in the following theorem. 

\begin{theorem}\label{teo_radical}
    Consider a directed graph $E$ and a field $K$. If we consider the path algebra $KE$, then $\rad(KE)$ is the two-sided ideal generated by the set of edges without return, that is,
    \begin{equation*}
        X = \{e \in E^1 \colon \nexists \lambda \in \PP(E) \text{ with } s(\lambda) = r(e), r(\lambda) = s(e)\}
    \end{equation*}
\end{theorem}
\begin{proof}
        Consider $I =  (X)$ the ideal generated by the set of edges $X$. First of all, Proposition \ref{prop_radical_path} directly implies that $I \subseteq \rad(KE)$. 

        Conversely, if we prove that every path without return is in $I$, then $\rad(KE) \subseteq I$ and the proof is complete. If we consider $\lambda$ a no return path with $\lambda = f_1f_2\cdots f_n \in \PP(E)$ such that there is no path $\mu \in \PP(E)$ with $s(\mu) = r(\lambda)$ and $r(\mu) = s(\lambda)$, then there is an edge $f_i$ with $i \in \{1,\ldots, n\}$ such that $f_i$ is a no return edge. If this were not true, then for every $i=1,2,\ldots,n$ there would exist a path $\mu_i \in \PP(E)$ such that $s(\mu_i) = r(f_i)$ and $r(\mu_i) = s(f_i)$. In this way, we could consider the path $\mu = \mu_n \mu_{n-1}\cdots \mu_1 \in \PP(E)$ which is a return path for $\lambda$ and this can not be possible.  In particular, $\lambda \in I$.    
\end{proof}

\begin{remark}
    If $E$ is a finite graph without cycles (a quiver) we denote $KE_+$ as the ideal generated by $E^1$. Recall that $KE_+ \cong \rad(KE)$.
\end{remark}

The fact that the Jacobson radical of a path algebra is an ideal generated by edges is really interesting. This means that the algebra $KE/\rad(KE)$ is isomorphic to a path algebra of a new graph. Furthermore, $KE/\rad(KE)$ is a semiprime algebra, in consequence, we have that, following Theorem \ref{teorema_estructura_semiprima},
\begin{equation}\label{descomp}
    KE/ \rad(KE) \cong\bigoplus_{i \in I} Kv_i  \oplus \bigoplus_{j \in J} KC_{n_j} \oplus \bigoplus_{l \in L} KF_l
\end{equation}
 where $v_i,C_{n_j}$ and $F_l$ are the strongly connected components of the graph $E$. In particular, $KF_l$ are primitive, $KC_{n_j}$ are prime and $Kv_i$ are simple path algebras.

As we know from Proposition \ref{prop_noeth_formal_triangular}, the Noetherian path algebras (left or right) are isomorphic to a upper triangular matrix algebras. In order to study its corresponding Jacobson radical we obtain the following result. 

\begin{proposition}\label{prop_radical_noetheriano}
Let us consider a directed graph $E$ with $A = KE$ its corresponding path algebra over the field $K$. If $A$ is a left Noetherian path algebra, following the description given in Proposition \ref{prop_noeth_formal_triangular}, the Jacobson radical of $A$ is
\begin{equation*}
    \rad(A) = \left ( \begin{array}{cc}
        0 & B \\
        0 & \rad(T)
    \end{array}\right )
\end{equation*}
with $\rad(T) \cong K\Sigma(E)_+$. Analogously, if $A$ is right Noetherian, then 
\begin{equation*}
    \rad(A) = \left ( \begin{array}{cc}
        \rad(S) & B \\
         0 & 0
    \end{array}\right )
\end{equation*}
with $\rad(S) \cong K\Sigma(E)_+$.
\end{proposition}
\begin{proof}
These results are immediate to prove thanks to Proposition \ref{prop_noeth_formal_triangular} and Corollary \cite[II.1]{radical_formal_triangular} which implies that 
\begin{equation*}
    \rad \left ( \begin{array}{cc}
        S & B \\
        0 & T
    \end{array}\right ) = \left ( \begin{array}{cc}
        \rad(S) & B \\
        0 & \rad(T)
    \end{array}\right ).  
\end{equation*}
Following Proposition \ref{prop_radical_path} we prove the rest. 
\end{proof}

Finally, after all the results obtained above, we are in condition to develop an structure theorem for the Noetherian (left and right) path algebras. We will make use of the concepts of socle of a path algebra (Proposition \ref{prop_zocalo_pathalgebra}) the radical of a path algebra (Theorem \ref{teo_radical}), together with Proposition \ref{prop_noeth_formal_triangular}.

\begin{theorem}[\bf Structure Theorem for left Noetherian path algebras]\label{teorema_structurenoetheriana}  If $E$ is a directed graph with $KE$ left (resp. right) Noetherian, then

\begin{equation*}
    KE/ \rad(KE) \cong\bigoplus_{i =1}^{n_0} Kv_i  \oplus \bigoplus_{j = 1}^s KC_{n_j}.
\end{equation*}
Let us consider two directed graphs $E$ and $F$, and denote by $s$ and $t$ the number of cycles in the graph $E$ and $F$ respectively and $n_i,m_j$ (for $i=1,\ldots,s$ and $j=1,\ldots,t$) the number of vertices in the corresponding cycle. If $KE$ and $KF$ are left (resp. right) Noetherian and $KE \cong KF$, then $|E^0| = |F^0|$, $s = t$ and there is a permutation $\sigma \in S_s$ that $ n_i = m_{\sigma(i)} \text{ for } i=1,\ldots, s$.
\end{theorem}
\begin{proof}
Making use of Equation \eqref{descomp} and Proposition \ref{prop_noetheriano} the first part is immediate.

Let us consider $A_1= KE$ and $A_2=KF$ left Noetherian algebras (for the right Noetherian algebras is analogous). Thanks to Proposition \ref{prop_noeth_formal_triangular} we have
\begin{equation*}
    A_1\cong \left (\begin{array}{cc}
        \bigoplus_{i=1}^s KC_{n_i} & B_1 \\
        0 & T_1 
    \end{array}\right ) \cong \left (\begin{array}{cc}
        \bigoplus_{j=1}^t KC_{m_j} & B_2 \\
        0 & T_2 
    \end{array}\right ) \cong A_2.
\end{equation*}
Following Proposition \ref{prop_radical_noetheriano} and Theorem \cite[0.1.4]{kasch1982modules} we have that $\rad(A_1)\cong \rad(A_2)$ and
\begin{equation*}
    \rad(A_1) \cong \left (\begin{array}{cc}
        0 & B_1 \\
        0 & \rad(T_1) 
    \end{array}\right ) \cong \left (\begin{array}{cc}
        0 & B_2 \\
        0 & \rad(T_2) 
    \end{array}\right ) \cong \rad(A_2).
\end{equation*}
If we consider $\overline{A_1} = A_1/\rad(A_1)$ and $\overline{A_2}= A_2/\rad(A_2)$ we have $\overline{A_1}\cong \overline{A_2}$ which means
\begin{equation*}
    \overline{A_1} \cong \left (\begin{array}{cc}
        \bigoplus_{i=1}^s KC_{n_i} & 0 \\
        0 & K^{n_0}
    \end{array}\right ) \cong \left (\begin{array}{cc}
        \bigoplus_{j=1}^t KC_{m_j} & 0 \\
        0 & K^{m_0}
    \end{array}\right ) \cong \overline{A_2}
\end{equation*}
with $n_0$ the number of vertices of $\Sigma(E)$ and $m_0$ the number of vertices of $\Sigma(F)$. Again, by Theorem \cite[9.1.4]{kasch1982modules}, we have that $\soc(\overline{A_1})\cong \soc(\overline{A_2})$ with $\soc(\overline{A_1}) = K^{n_0}$ and $\soc(\overline{A_2}) = K^{m_0}$  (by Corollary \ref{corolario_zocalo_pathalgebra}) which means that $n_0=m_0$. Then, we have that $\bigoplus_{i=1}^s KC_{n_i} \cong  \bigoplus_{j=1}^t KC_{m_j}$. 
By Wedderburn-Artin Theorem \cite[p. 204]{jacobson2012basic} we have that $s = t$ and there is a permutation $\sigma \in S_t$ such that $n_i = m_{\sigma(i)}$. 
\end{proof}

In order to describe the structure of left Noetherian path algebras $KE$, we first approach its quotient $\overline{KE}=KE/\rad(KE)$. Equivalently, we consider a left Noetherian Jacobson semisimple path algebra $KE$. Then  $KE=\soc(KE)\oplus\ann(\soc(KE))$ and the socle is a finite direct sum $K^n$ (with componentwise product). 
A first invariant under isomorphism is the pair
$(\soc(KE),\ann(\soc(KE))$. We have $\ann(\soc(KE))=\oplus_i KC_{n_i}$ (finite direct sum).
Since the central closure of each $KC_{n_i}$ is 
$\widehat{KC_{n_i}}\cong M_{n_i}(K(x))$ we have that the central closure of $\ann(\soc(KE))$ is an Artinian semisimple $K(x)$-algebra and, in fact, is isomorphic to 
the (finite) direct sum $\oplus_{i}M_{n_i}(K(x))$. 
Therefore, an isomorphism of left Noetherian J-semisimple path algebras $KE\cong KF$ directly leads to an isomorphism of Artinian semisimple $K(x)$-algebras. Consequently the tuples $(n_1,\ldots, n_s)$, $(m_1,\ldots, m_t)$ of both algebras have the same length and up to a permutation the $n_i$'s and the $m_j$'s coincide. Summarizing: the tuple
$(n,n_1,\ldots,n_k)$ determine completely $KE$ and 
$KE=K^n\oplus KC_{n_1}\oplus\cdots\oplus KC_{n_k}$. Furthermore, $n=\dim_K(\soc(KE))$ and each $n_i^2$ is the dimension (as $K(x)$-vector space) of the simple components of the Artinian semisimple part $\widehat{\ann(\soc(KE))}$. Alternatively, one can see that each $n_i$ is the uniform dimension of $KC_{n_i}$.

\section{Declarations}

\subsection*{Ethical Approval:}

This declaration is not applicable.

\subsection*{Conflicts of interests/Competing interests:} We have no conflicts of interests/competing interests to disclose.

\subsection*{Authors' contributions:}

All authors contributed equally to this work. 

\subsection*{Data Availability Statement:} The authors confirm that the data supporting the findings of this study are available within the article.
\vskip 1cm

\section*{Acknowledgements}{The authors are supported by the Spanish Ministerio de Ciencia e Innovaci\'on   through projects  PID2019-104236GB-I00 and PID2023-152673NB-I00 and by the Junta de Andaluc\'{i}a  through projects  FQM-336 and UMA18-FEDERJA-119,  all of them with FEDER funds. The third author is supported by a Junta de Andalucía PID fellowship no. PREDOC\_00029.}

\bibliographystyle{plain}
\bibliography{ref}

\begin{thebibliography}{10}

\bibitem{lpa}
Gene Abrams, Pere Ara, and Mercedes~Siles Molina.
\newblock {\em Leavitt path algebras}, volume 2191.
\newblock Springer, 2017.

\bibitem{amitsur_levitzki}
Avraham.~S. Amitsur and Jacob Levitzki.
\newblock Minimal identities for algebras.
\newblock {\em Proceedings of the American Mathematical Society}, 1(4):449--463, 1950.

\bibitem{anquela1998local}
Jos{\'e}~A. Anquela and Teresa Cort{\'e}s.
\newblock Local-to-global inheritance of primitivity in {J}ordan algebras.
\newblock {\em Archiv der Mathematik}, 70(3):219--227, 1998.

\bibitem{baxter_central_closure}
Willard~E. Baxter and Wallace~S. Martindale~III.
\newblock Central closure of semiprime non-associative rings.
\newblock {\em Communications in Algebra}, 7(11):1103--1132, 1979.

\bibitem{bergman1964ring}
George~M. Bergman.
\newblock A ring primitive on the right but not on the left.
\newblock {\em Proceedings of the American Mathematical Society}, 15(3):473--475, 1964.

\bibitem{centralizadores}
George~M. Bergman.
\newblock Centralizers in free associative algebras.
\newblock {\em Transactions of the American Mathematical Society}, 137:327--344, 1969.

\bibitem{PI_algebras}
Matej Bre{\v{s}}ar.
\newblock An alternative approach to the structure theory of pi-rings.
\newblock {\em Expositiones Mathematicae}, 29(1):159--164, 2011.

\bibitem{centro_algebra_caminos}
María~G. Corrales~García, Dolores Martín~Barquero, Cándido Martín~González, Mercedes Siles~Molina, and José~F. Solanilla~Hernández.
\newblock Centers of path algebras, {C}ohn and {L}eavitt path algebras.
\newblock {\em Bulletin of the Malaysian Mathematical Sciences Society}, 40(4):1745--1767, 2017.

\bibitem{quivlecs}
William Crawley-Boevey.
\newblock Lectures on representations of quivers.
\newblock \url{https://www.math.uni-bielefeld.de/~wcrawley/quivlecs.pdf}.

\bibitem{goodear}
Kenneth~R. Goodearl and Robert~Breckenridge Warfield~Jr.
\newblock {\em An introduction to noncommutative Noetherian rings}, volume~61.
\newblock Cambridge university press, 2004.

\bibitem{radical_formal_triangular}
Ahmad Haghany and Kalathoor Varadarajan.
\newblock Study of formal triangular matrix rings.
\newblock {\em Communications in Algebra}, 27(11):5507--5525, 1999.

\bibitem{jacobson1956structure}
Nathan Jacobson.
\newblock {\em Structure of rings}, volume~37.
\newblock American Mathematical Soc., 1956.

\bibitem{jacobson2012basic}
Nathan Jacobson.
\newblock {\em Basic Algebra II: Second Edition}.
\newblock Dover Books on Mathematics. Dover Publications, 2012.

\bibitem{kasch1982modules}
Friedrich Kasch.
\newblock {\em Modules and rings}, volume~17.
\newblock Academic press, 1982.

\bibitem{lam1991first}
Tsit-Yuen Lam.
\newblock {\em A first course in noncommutative rings}, volume 131.
\newblock Springer, 1991.

\bibitem{prop_semiprimo_mercedes}
Mercedes~Siles Molina.
\newblock Algebras of quotients of path algebras.
\newblock {\em Journal of Algebra}, 319(12):5265--5278, 2008.

\bibitem{paulsmith}
Paul Smith.
\newblock Representations of quivers.
\newblock \url{https://sites.math.washington.edu/~smith/Teaching/513nag/notes6.pdf}.

\end{thebibliography}

\end{document}